\documentclass{amsart}
\usepackage{latexsym, amssymb, amsmath, amsthm}
\usepackage[T1]{fontenc}
\usepackage{lmodern}
\usepackage{enumerate}
\usepackage{mathabx}
\usepackage{csquotes}
\usepackage[mathscr]{euscript}
\usepackage{parskip}
\usepackage{thmtools, thm-restate}
\usepackage{bussproofs}
\usepackage{xcolor}

\newcommand\blfootnote[1]{%
  \begingroup
  \renewcommand\thefootnote{}\footnote{#1}%
  \addtocounter{footnote}{-1}%
  \endgroup
}
\makeatletter
\newlength\knuthian@fdfive
\def\mathpal@save#1{\let\was@math@style=#1\relax}
\def\utilde#1{\mathpalette\mathpal@save
              {\setbox124=\hbox{$\was@math@style#1$}%
\setbox125=\hbox{$\fam=3\global\knuthian@fdfive=\fontdimen5\font$}
\setbox125=\hbox{$\widetilde{\vrule height 0pt depth 0pt width \wd124}$}%
               \baselineskip=1pt\relax
               \lineskiplimit=\z@\relax
               \lineskip=1pt\relax
               \vtop{\copy124\copy125\vskip -\knuthian@fdfive}}}
\makeatother

\usepackage{pifont}

\declaretheorem[numberwithin=section]{theorem}
\newtheorem{lemma}[theorem]{Lemma}
\newtheorem{corollary}[theorem]{Corollary}
\newtheorem{proposition}[theorem]{Proposition}

\theoremstyle{definition}
\newtheorem{definition}[theorem]{Definition}
\theoremstyle{remark}
\newtheorem{remark}[theorem]{Remark}

\newbox\gnBoxA
\newdimen\gnCornerHgt
\setbox\gnBoxA=\hbox{$\ulcorner$}
\global\gnCornerHgt=\ht\gnBoxA
\newdimen\gnArgHgt
\def\gnmb #1{%
\setbox\gnBoxA=\hbox{$#1$}%
\gnArgHgt=\ht\gnBoxA%
\ifnum     \gnArgHgt<\gnCornerHgt \gnArgHgt=0pt%
\else \advance \gnArgHgt by -\gnCornerHgt%
\fi \raise\gnArgHgt\hbox{$\ulcorner$} \box\gnBoxA %
\raise\gnArgHgt\hbox{$\urcorner$}}

\newcommand{\defiff}{\stackrel{\mbox{\scriptsize $\textrm{def}$}}{\iff}}

\begin{document}

\title{Reducing $\omega$-model reflection to iterated syntactic reflection}
\author{Fedor Pakhomov}
\thanks{Research of Fedor Pakhomov is supported by FWO Senior Postdoctoral Fellowship, project 1283021N}
\author{James Walsh}
\address{Ghent University and Steklov Mathematical Institute of Russian Academy of Sciences}
\email{fedor.pakhomov@ugent.be}
\address{Sage School of Philosophy, Cornell University}
\email{jameswalsh@cornell.edu}

\begin{abstract}
    In mathematical logic there are two seemingly distinct kinds of principles called ``reflection principles.'' Semantic reflection principles assert that if a formula holds in the whole universe, then it holds in a set-sized model. Syntactic reflection principles assert that every provable sentence from some complexity class is true. In this paper we study connections between these two kinds of reflection principles in the setting of second-order arithmetic. We prove that, for a large swathe of theories, $\omega$-model reflection is equivalent to the claim that arbitrary iterations of uniform $\Pi^1_1$ reflection along countable well-orderings are $\Pi^1_1$-sound. This result yields uniform ordinal analyses of theories with strength between $\mathsf{ACA}_0$ and $\mathsf{ATR}$. The main technical novelty of our analysis is the introduction of the notion of the proof-theoretic dilator of a theory $T$, which is the operator on countable ordinals that maps the order-type of $\prec$ to the proof-theoretic ordinal of $T+\mathsf{WO}(\prec)$. We obtain precise results about the growth of proof-theoretic dilators as a function of provable $\omega$-model reflection. This approach enables us to simultaneously obtain not only $\Pi^0_1$, $\Pi^0_2$, and $\Pi^1_1$ ordinals but also reverse-mathematical theorems for well-ordering principles.
\end{abstract}

\maketitle

\blfootnote{Thanks to Antonio Montalb\'{a}n for helpful discussions of this material and for comments on drafts. Thanks to the referee for helpful comments.}

\section{Introduction}

Two types of principles are commonly called \emph{reflection principles} in mathematical logic. First, according to \emph{semantic} reflection principles, whatever is true in the universe holds in some set-sized model. The thought here is that structures within the universe reflect what is happening in the universe. Second, according to \emph{syntactic} reflection principles, whatever is provable is true. The thought here is that we should endorse these principles by reflecting on the soundness of our axioms and inference rules. Kreisel and L\'{e}vy \cite{kreisel1968reflection} wrote that they could not agree whether this terminological coincidence was ``merely a pun.''

In this paper we establish results relating both types of reflection principles in second-order arithmetic. In particular, we establish conservation theorems reducing $\omega$-model reflection principles to iterated syntactic reflection principles. There is a thorough proof-theoretic understanding of the latter, e.g., in terms of ordinal analysis. Accordingly, these reductions yield proof-theoretic analyses of $\omega$-model reflection principles. In particular, we use these reductions to  uniformly analyze theories between $\mathsf{ACA}_0$ and $\mathsf{ATR}$ in terms of both proof-theoretic ordinals and well-ordering principles.

We will be concerned in this paper with iterations of reflection along arbitrary (potentially non-recursive) well-orderings. We formally define these iterations using the language $\mathbf{L}_2$ that extends the standard language of second-order arithmetic with set-constants $C_X$ for all sets $X$. $\mathbf{L}_2$ formulas can be encoded as sets, and definitions of $\mathbf{L}_2$ theories and $\mathbf{L}_2$'s standard syntactic complexity classes can be given in $\mathsf{ACA}_0$. Accordingly, throughout this paper we formalize our results in $\mathsf{ACA}_0$.


The main syntactic reflection principle we consider, $\mathbf{\Pi}^1_n$-$\mathsf{RFN}(T)$, informally says ``all $\mathbf{\Pi}^1_n$ theorems of $T$ are true.'' We will also be interested in the theories $\mathbf{\Pi}^1_n$-$\mathbf{R}^\alpha(T)$ that result from iterating this principle along well-orderings $\alpha$. We give precise definitions of these theories via G\"{o}del's fixed point lemma in \textsection \ref{linear_orders}. Informally, one can think of them as defined inductively, according to the following equation:
$$\mathbf{\Pi}^1_n\text{-}\mathbf{R}^{\alpha}(T):= T + \big\{ \mathbf{\Pi}^1_n\text{-}\mathsf{RFN}\big(T+ \mathbf{\Pi}^1_n\text{-}\mathbf{R}^\beta (T)\big) : \beta<\alpha \big\}$$

We use the notation $\mathbf{\Pi}^1_n\text{-}\mathbf{R}^{\textsc{on}}(T)$ in place of 
$$\forall \alpha \big(\mathsf{WO}(\alpha) \rightarrow \mathbf{\Pi}^1_n\textrm{-}\mathsf{RFN}(\mathbf{\Pi}^1_n\textrm{-}\mathbf{R}^\alpha(T))\big).$$

We will also consider $\omega$-model reflection principles, according to which all sets are encoded in $\omega$-models of $T$. An $\omega$-model is an $L_2$ structure whose first-order part is $\mathbb{N}$ and whose second-order part is some subset of $\mathcal{P}(\mathbb{N})$.


Our main theorem is the following:

\begin{theorem}[$\mathsf{ACA}_0$]\label{reduction_for_omega_models}
For any $\mathbf{\Pi}^1_2$ axiomatized theory $T$, the following are equivalent:
\begin{enumerate}
\item Every set is contained in an $\omega$-model of $T$.
\item $\mathbf{\Pi}^1_1\text{-}\mathbf{R}^{\textsc{on}}(T)$
\end{enumerate}
\end{theorem}

Theorem \ref{reduction_for_omega_models} provides a reduction of $\omega$-model reflection to iterated syntactic reflection. This reduction is desirable because of the distinct roles the two types of principles play in second-order arithmetic. On the one hand, $\omega$-model reflection principles are well-known in reverse mathematics, since many theories of interest can be axiomatized in terms of $\omega$-model reflection principles. On the other hand, iterated syntactic reflection principles are widely studied in ordinal analysis because of the systematic connections between iterated reflection and proof-theoretic ordinals (see, e.g., \cite{pakhomov2018reflection}). Thus, Theorem \ref{reduction_for_omega_models} opens the path to a systematic connection between $\omega$-model reflection and ordinal analysis.

In fact, the proof of Theorem \ref{reduction_for_omega_models} delivers a more general theorem.  The semantic reflection principle we work with is $\mathbf{\Pi}^1_n\textrm{-}\omega\mathsf{RFN}(T)$, which informally says that ``any $\mathbf{\Pi}^1_n$ sentence that holds in all $\omega$-models of $T$ is true.'' Note that $\mathbf{\Pi}^1_1\textrm{-}\omega\mathsf{RFN}(T)$ and $\mathbf{\Pi}^1_2\textrm{-}\omega\mathsf{RFN}(T)$ are both equivalent to the claim that every set is contained in an $\omega$-model of $T$.  We are able to generalize Theorem \ref{reduction_for_omega_models} as follows:

\begin{theorem}\label{main_theorem}
Suppose $n>0$. Then $\mathsf{ACA}_0$ proves that for any $\mathbf{\Pi}^1_{n+1}$-axiomatizable $T$, the following are equivalent:
\begin{enumerate}
\item $\mathbf{\Pi}^1_n\textrm{-}\omega\mathsf{RFN}(T)$
\item $\mathbf{\Pi}^1_n\text{-}\mathbf{R}^{\textsc{on}}(T)$
\end{enumerate}
\end{theorem}

An interesting feature of Theorem \ref{main_theorem} is that on the one hand it appears similar to the result of J\"ager and Strahm \cite{jager1999bar} about the equivalence of $\mathbf{\Pi}^1_{n+2}\textrm{-}\omega\mathsf{RFN}(T)$ and $\Pi^1_{n}\textrm{-}\mathsf{BI}$, for $n\ge 1$.  And on the other hand Theorem \ref{main_theorem} appears similar to reduction properties for systems of first-order arithmetic, see \cite{schmerl1979fine, beklemishev2003proof}. In \cite{pakhomov2018reflection} the authors proved a Schmerl-style theorem for iterated $\Pi^1_1$ reflection and used it to establish a systematic connection between iterated $\Pi^1_1$ reflection and $\Pi^1_1$ proof-theoretic ordinals. Theorem \ref{main_theorem} extends these results in two important ways. First, it extends them to treat semantic reflection principles, namely $\omega$-model reflection principles. Second, it extends them to yield information sufficient not only for $\Pi^1_1$ proof-theoretic analysis (i.e., the calculation of $\Pi^1_1$ proof-theoretic ordinals), but also for the $\Pi^1_2$ proof-theoretic analysis of theories.

With respect to $\Pi^1_1$ ordinal analysis, we use Theorem \ref{main_theorem} to uniformly prove the following known results, where $|T|_{\mathbf{\Pi}^1_1}$ is the proof-theoretic ordinal of $T$:

\begin{theorem}
\begin{enumerate}
    \item $|\mathsf{ACA}_0^+|_{\mathbf{\Pi}^1_1}=\phi_2(0)$;
    \item $|\Sigma^1_1$-$\mathsf{AC}|_{\mathbf{\Pi}^1_1}=|\Pi^1_2$-$\mathsf{RFN}^{\varepsilon_0}(\Sigma^1_1$-$\mathsf{AC}_0)|_{\mathbf{\Pi}^1_1}=\phi_{\varepsilon_0}(0)$
    \item $|\mathsf{ATR}_0|_{\mathbf{\Pi}^1_1}=\Gamma_0$.
    \item $|\mathsf{ATR}|_{\mathbf{\Pi}^1_1}=\Gamma_{\varepsilon_0}$.
\end{enumerate}
\end{theorem}

$\Pi^1_2$ proof theory, pioneered by Girard \cite{girard1981pi12}, is concerned with \emph{dilators}, certain well-behaved functions on the ordinals. In this paper we introduce the notion of the \emph{dilator of a theory}, which is roughly a function encapsulating the closure conditions that a theory imposes on the ordinals. More formally, we use the following definition:
\begin{definition}
The \emph{proof-theoretic dilator} of a theory $T$ is the function $\omega_1\cup\{\infty\}\to\omega_1\cup\{\infty\}$: $$|\alpha|\longmapsto |T+\mathsf{WO}(\dot\alpha)|_{\mathbf{\Pi}^1_1}$$ where $\alpha$ ranges over countable linear orders. We write $|T|_{\mathbf{\Pi}^1_2}$ to denote the proof-theoretic dilator of $T$.
\end{definition}

In \cite{pakhomov2018reflection} the authors developed a systematic connection between iterated $\Pi^1_1$ reflection and $\Pi^1_1$ ordinal analysis, including the following theorem (stated using the terminology of this paper):
\begin{theorem}\label{previous_paper}
$|\mathbf{\Pi}^1_1$-$\mathbf{R}^\alpha(\mathsf{ACA}_0)|_{\Pi^1_1}=\varepsilon_\alpha$.
\end{theorem}
Theorem \ref{previous_paper} specifies how proof-theoretic ordinals of a theory grows as a function of the amount of $\mathbf{\Pi}^1_1$ reflection it proves. In this paper we develop a similar systematic connection between iterated $\omega$-model reflection and dilators of theories. In particular, we pin down how theories' dilators climb the Veblen hierarchy as a function of the amount of $\omega$-model reflection postulated. Before stating this connection, we introduce some notation. For linear orders $\alpha,\beta,\gamma$ we write $\phi_\alpha^+(\beta)$ to denote the standard notation system for the least ordinal strictly above $\beta$ that is a value of $\phi_\alpha$ function. And we write $\phi_\alpha^{+\gamma}(\beta)$ to denote the standard notation system for the $\gamma^{th}$ ordinal above $\beta$ that is a value of $\phi_\alpha$-function. We characterize iterated $\mathbf{\Pi}^1_2$ reflection via dilators as follows:

\begin{theorem}[$\mathsf{ACA}_0$]\label{iter-intro}
Let $T$ be a $\mathbf{\Pi}^1_2$-axiomatizable theory such that $|T|_{\mathbf{\Pi}^1_2}=|\phi_\alpha^+|$, for some linear order $\alpha$. Then for any $\beta$ we have $|\mathbf{\Pi}^1_2\mbox{-}\mathbf{R}^\beta(T)|_{\mathbf{\Pi}^1_2}=|\phi_{\alpha}^{+\omega^\beta}|$.
\end{theorem}

Combining Theorem \ref{iter-intro} with Theorem \ref{main_theorem} yields the following, where we write $\mathbf{\Pi}^1_1\mbox{-}\omega\mathbf{R}^\alpha(T)$ for the result of iterating $\mathbf{\Pi}^1_1\mbox{-}\omega\mathsf{RFN}$ along $\alpha$ starting with $T$:

\begin{theorem}[$\mathsf{ACA}_0$]\label{iter_in_intro}
For any linear order $\alpha$
 $$|\mathbf{\Pi}^1_1\mbox{-}\omega\mathbf{R}^\alpha(\mathsf{ACA}_0)|_{\mathbf{\Pi}^1_2}=|\phi_{1+\alpha}^+|.$$
\end{theorem}

This latter result is useful in the reverse mathematics of well-ordering principles. In this corner of reverse mathematics, a $\Pi^1_2$-axiomatized theory $T$ is shown to be equivalent to a comprehension principle related to the term system of $T$'s proof-theoretic ordinal. The classic result in this area, due to Girard \cite{girard1987proof}, is the following:

\begin{theorem}[Girard] Over $\mathsf{RCA}_0$, $\mathsf{ACA}_0$ is equivalent to the well-ordering principle $\forall \alpha \big( \mathsf{WO}(\alpha) \rightarrow \mathsf{WO}(\omega^\alpha) \big)$.
\end{theorem}

In recent years there has been a renewed interest in such well-ordering principles, and they have been pursued by a variety of recursion-theoretic and proof-theoretic methods \cite{marcone2011veblen, afshari2009reverse}. Two of the results produced by this line of research are the following:

\begin{theorem}
[Marcone--Montalb\'{a}n]\label{marcone-mont}
Over $\mathsf{RCA}_0$, $\mathsf{ACA}_0^+$ is equivalent to the well-ordering principle $\forall \alpha\Big(\mathsf{WO}(\alpha)\to  \mathsf{WO}\big(\phi_1(\alpha)\big)\Big).$
\end{theorem}

\begin{theorem}[H. Friedman--Montalb\'{a}n--Weiermann]\label{fmw_intro}
Over $\mathsf{RCA}_0$, $\mathsf{ATR}_0$ is equivalent to the well-ordering principle $\forall \alpha\Big(\mathsf{WO}(\alpha)\to  \mathsf{WO}\big(\phi_\alpha(0)\big)\Big).$
\end{theorem}

We use Theorem \ref{iter_in_intro} to prove a number of reverse mathematical results of this sort. In particular, we provide new proofs of Theorems \ref{marcone-mont} and \ref{fmw_intro} (over the base theory $\mathsf{ACA}_0$).

In some ways this work is similar to the habilitation thesis of Probst \cite{probst2017modular}. Probst carried out ordinal analyses of metapredicative theories (roughly, those theories in the range that we consider) through the analysis of reflection operators. In particular, he considered the theories generated by $\Pi^1_n$ $\omega$-model reflection operators. In technical respects his analysis was quite different from ours, insofar as it was based on elimination of $\Pi^1_n$-cuts from infinitary derivations. By contrast, we isolate the consideration of cut-free infinitary derivations to the proof of Theorem \ref{main_theorem} and elsewhere use reflexive induction and reduction principles for reflection principles.

Here is our plan for the rest of the paper. In \textsection \ref{preliminaries} we cover a number of preliminaries. We present the class language $\mathbf{L}_2$ for second-order arithmetic. We then define its syntactic complexity classes, their attendant reflection principles, and the iterations thereof. We prove a number of lemmas about the basic properties of iterated reflection principles, including a reduction principle for iterated $\omega$-model reflection. We also prove a theorem relating iterations of $\omega$-model reflection and iterates of the Turing jump. In \textsection \ref{equivalent_forms} we define an infinitary proof system, an $\omega$-proof-system for $\mathbf{L}_2$. We prove that this proof system is sound and complete with respect to $\omega$-models, which is crucial for our main results. In \textsection \ref{reduction} we prove the main theorems of our paper. In particular, we prove Theorem \ref{main_theorem}, a reduction of $\omega$-model reflection to iterated syntactic reflection. In \textsection \ref{pt_dilators} we introduce the notion of a proof-theoretic dilator, and prove Theorem \ref{iter-intro} and Theorem \ref{iter_in_intro}, which characterize iterated reflection in terms of dilators. We then turn to applications. In \textsection \ref{choice} we characterize $\mathsf{ATR}_0$ in terms of reflection over both $\mathsf{ACA}_0$ and $\Sigma^1_1\text{-}\mathsf{AC}_0$. In \textsection \ref{well-ordering-principles} we provide uniform calculations of $\Pi^1_1$ proof-theoretic ordinals of theories between $\mathsf{ACA}_0$ and $\mathsf{ATR}$. We also characterize $\mathsf{ACA}_0^+$ and $\mathsf{ATR}_0$ as well-ordering principles. It is worth noting that \textsection \ref{choice} and \textsection \ref{well-ordering-principles} jointly contain new proofs of all of the major proof-theoretic meta-theorems about $\mathsf{ATR}_0$.


\section{Preliminaries}\label{preliminaries}

In this section we will outline our treatment of theories, languages, complexity classes, and reflection principles. Our base system is the system $\mathsf{ACA}_0$. Since this theory is finitely axiomatizable we identify it with a sentence giving its finite axiomatization. We will also be interested in the systems $\Sigma^1_1\mbox{-}\mathsf{AC}_0$, $\mathsf{ACA}_0^+$, and $\mathsf{ATR}_0$. Since these theories are finitely axiomatizable we identify them with sentences giving their finite axiomatization.

Throughout this paper we restrict our attention to theories extending $\mathsf{ACA}_0$. So whenever we make a claim about ``every theory $T$,'' we mean ``every theory extending $\mathsf{ACA}_0$.''

\subsection{Languages and Complexity Classes}

In this chapter we will study reflection principles for formulas with set parameters. In the study of reflection and provability in first-order arithmetic it is common to study provability for formulas with number parameters; note the parameter in the expression $$\forall x \big( \mathsf{Prv}_{\mathsf{PA}}(\gnmb{\varphi(\dot{x})}) \rightarrow \varphi(x)\big).$$ Here $\gnmb{\varphi(\dot{x})}$ denotes the G\"odel number of the formula, $\varphi(\underline{x})$, where $\underline{x}$ is the numeral $$\underbrace{S(\ldots S}\limits_{\mbox{\scriptsize $x$ times}}(0)\ldots)).$$ This strategy is not available if we want to formalize claims about the provability of formulas with set parameters. Since there are no numerals for sets of natural numbers, we need to use a different approach to pass second-order variables inside provability predicates/reflection principles.

We write $L_2$ to denote the standard language of second-order arithmetic. We write $\mathbf{L}_2$ to denote the extension of $L_2$ with set-constants $C_X$ for all sets $X$. From the external perspective $\mathbf{L}_2$ is a continuum-sized language. However, formulas of $\mathbf{L}_2$ can be encoded by sets and reasoned about within $\mathsf{ACA}_0$. We will use the rest of this subsection to explain how this is accomplished.  

$L_2$ formulas are finitary objects and are encoded by natural numbers. The code for an $\mathbf{L}_2$ formula $\varphi(C_{Y_1},...,C_{Y_n}, \vec{x})$ is a pair $\big(\varphi(X_1,...,X_n, \vec{x}), \langle Y_1,...,Y_n \rangle \big)$ where $\varphi(X_1,...,X_n, \vec{x})$ is (a code for) an $L_2$-formula  and $\langle Y_1,...,Y_n \rangle$ is a sequence of sets. Note that whereas $L_2$ formulas are encoded by numbers, $\mathbf{L}_2$ formulas are encoded by sets.

Standard manipulations of (codes of) $L_2$ formulas (e.g., forming conjunctions, performing substitutions, etc.) is totally finitary and thus can be carried out in $\mathsf{ACA}_0$ (indeed, in much weaker theories). Analogous manipulations of (codes of) $\mathbf{L}_2$ formulas is carried out on sets rather than on numbers. Nevertheless, $\mathsf{ACA}_0$ can carry out these sorts of manipulations. The code of a formula formed, e.g., by conjunction is \emph{arithmetic} in the codes of the conjuncts.

For any formula $\varphi(\vec{X},\vec{x},y)$, there is a function $\mathsf{cmp}_{\varphi,y}$ which maps $\vec{X},\vec{x}$ to $\big\{y : \varphi(\vec{X}, \vec{x},y)\big\}$. Arithmetical comprehension terms are terms built from functions $\mathsf{cmp}_{\varphi,y}$, where $\varphi\in \Pi^0_\infty$. Elementary comprehension terms are terms built from functions $\mathsf{cmp}_{\varphi,y}$, where $\varphi\in \Delta^0_0$ (note that . The functions that manipulate (codes of) $\mathbf{L}_2$ formulas are expressible as arithmetical comprehension terms (but not as elementary comprehension terms, since we need to check the equality of sets). For instance, the code of a conjunction is the output of the comprehension function corresponding to an arithmetic operation applied to the codes of the conjuncts. Given a formula $$\varphi(X_1,\ldots,X_n,y_1,\ldots,y_m)\in L_2$$ the expression $$\gnmb{\varphi(\dot X_1,\ldots,\dot X_n,\dot y_1,\ldots,\dot y_m)}$$ is the term (built using the definable comprehension functions) denoting the code of the formula $$\varphi(C_{X_1},\ldots,C_{X_n},\underline{y_1},\ldots,\underline{y_m}).$$ Since we usually will not consider codes for formulas with free variables, to simplify our notation, the expression $\varphi\Big(\psi(\vec{X},\vec{y})\Big)$ will serve as shorthand for  $\varphi\Big(\gnmb{\psi(\dot{\vec{X}},\dot{\vec{y}})}\Big)$.

As usual we write $\Pi^1_0=\Sigma^1_0$ ($\mathbf{\Pi}^1_0=\mathbf{\Sigma}^1_0$) to denote the class of $L_2$-formulas ($\mathbf{L}_2$-formulas) without second-order quantifiers. The class $\Pi^1_{n+1}\subseteq L_2$ ($\mathbf{\Pi}^1_{n+1}\subseteq \mathbf{L}_2$) consists of all formulas of the form $\forall \overrightarrow{Xx}\;\varphi$, where $\varphi\in \Sigma^1_n$ ($\varphi\in \mathbf{\Sigma}^1_n$) and $\overrightarrow{Xx}$ is a vector of variables that could contain both first and second order variables. The class $\Sigma^1_{n+1}\subseteq L_2$ ($\mathbf{\Sigma}^1_{n+1}\subseteq \mathbf{L}_2$) consists of all formulas of the form $\exists \overrightarrow{Xx}\;\varphi$, where $\varphi\in \Pi^1_n$ ($\varphi\in \mathbf{\Pi}^1_n$) and $\overrightarrow{Xx}$ is some vector of variables that could contain both first and second-order variables.

In second-order arithmetic it is useful to work with countable sets of sets of naturals. To do this we represent a countable set $\mathcal{S}\subseteq \mathcal{P}(\mathbb{N})$ by a code of a countable sequence $\langle S_i\subseteq \mathbb{N}\mid i\in A\rangle$, $A\subseteq\mathbb{N}$ such that $\mathcal{S}=\{S_i\mid i\in A\}$. Formally, we use the predicate $X\dot\in Y$ that says: $$\exists z\;\Big(\langle z,0\rangle\in Y \land \forall x\big(x\in X\mathrel{\leftrightarrow}  \langle z,x+1\rangle\in Y\big)\Big).$$

Inside $\mathsf{ACA}_0$ we work with countable $\mathbf{L}_2$-theories represented by sets $T$ treated as codes for their set of axioms. The provability predicate $\mathsf{Prv}(T,\varphi)$ expresses that $T$ is an $\mathbf{L}_2$-theory, $\varphi$ is an $\mathbf{L}_2$-formula, and there is a proof $P$ of $\varphi$ in first-order logic such that all non-logical axioms in $P$ are from $T$. Note that here the proof $P$ by necessity is encoded by a set. However, $\mathsf{Prv}$ is equivalent to a $\Pi^1_0$ formula.

\subsection{Reflection Principles}

A standard construction allows us to define in $\mathsf{ACA}_0$ partial truth definitions $\mathsf{Tr}_{\mathbf{\Pi}^1_n}(X)$, for the classes of formulas $\mathbf{\Pi}^1_n$. Here for any $\Pi^1_n$-formula $\varphi(\vec{X},\vec{y})$ we have that
$$\mathsf{ACA}_0\vdash \forall \vec{X},\vec{y}\Big (\varphi(\vec{X},\vec{y})\mathrel{\leftrightarrow} \mathsf{Tr}_{\mathbf{\Pi}^1_n}\big(\varphi(\vec{X},\vec{y})\big)\Big)$$
Note that the formulas $\mathsf{Tr}_{\mathbf{\Pi}^1_n}(X)$ are $\Pi^1_n$-formulas. We have truth definitions $\mathsf{Tr}_{\mathbf{\Sigma}^1_n}(X)$, for $n\ge 1$ with analogous properties as well.

For a theory $T$ we put $\mathbf{\Pi}^1_n\mbox{-}\mathsf{RFN}(T)$, $n\ge 1$, to be the $\mathbf{L}_2$-sentence $$\forall \varphi\in \mathbf{\Pi}^1_n\;(\mathsf{Prv}(T,\varphi)\to \mathsf{Tr}_{\mathbf{\Pi}^1_n}(\varphi)).$$
Or equivalently (over $\mathsf{ACA}_0$) we could reformulate $\mathbf{\Pi}^1_n\mbox{-}\mathsf{RFN}(T)$ as $\mathbf{\Sigma}^1_n\mbox{-}\mathsf{Con}(T)$, which expresses that $T$ is consistent with any true $\mathbf{\Sigma}^1_n$ sentence:
$$\forall \varphi\in \mathbf{\Sigma}^1_n\;(\mathsf{Tr}_{\mathbf{\Sigma}^1_n}(\varphi)\to\mathsf{Con}(T+\varphi)).$$

As the name suggests, $\mathbf{\Sigma}^1_n\mbox{-}\mathsf{Con}(T)$ is a consistency-like operator; thus, we can formulate its dual provability-like predicate. Namely, we define $$\mathbf{\Sigma}^1_n\mbox{-}\mathsf{Prv}(T,\varphi):=\lnot \mathbf{\Sigma}^1_n\mbox{-}\mathsf{Con}(T+\lnot\varphi).$$ The following argument demonstrates that $\mathbf{\Sigma}^1_n\mbox{-}\mathsf{Prv}(T,\varphi)$ expresses that $\varphi$ is provable from axioms of $T$ and one true $\mathbf{\Sigma}^1_n$-sentence:
\begin{flalign*}
\mathbf{\Sigma}^1_n\mbox{-}\mathsf{Prv}(T,\varphi) &\equiv \lnot \mathbf{\Sigma}^1_n\mbox{-}\mathsf{Con}(T+\lnot\varphi)\\
&\equiv \neg \forall \psi\in\mathbf{\Sigma}^1_n \big( \mathsf{Tr}_{\mathbf{\Sigma}^1_n}(\psi)\rightarrow\mathsf{Con}(T+\neg\varphi+\psi)\big)\\
&\equiv \exists \psi\in\mathbf{\Sigma}^1_n\big(\mathsf{Tr}_{\mathbf{\Sigma}^1_n}(\psi)\wedge\neg\mathsf{Con}(T+\neg\varphi+\psi)\big)\\
&\equiv\exists\psi\in\mathbf{\Sigma}^1_n\big(\mathsf{Tr}_{\mathbf{\Sigma}^1_n}(\psi)\wedge\mathsf{Prv}(T+\psi,\varphi)\big)
\end{flalign*}

In the context of our paper the most important equivalence notion on theories will be equivalence up to $\mathbf{\Sigma}^1_1$-provability. We first define the notion of one theory being included in another up to $\mathbf{\Sigma}^1_1$ provability.
$$T\sqsubseteq^{\mathbf{\Sigma}^1_1} U\defiff \forall \varphi \in \mathbf{L}_2( \mathbf{\Sigma}^1_1\mbox{-}\mathsf{Prv}(T,\varphi)\mathrel{\rightarrow} \mathbf{\Sigma}^1_1\mbox{-}\mathsf{Prv}(U,\varphi)).$$
Two theories are equivalent up to $\mathbf{\Sigma}^1_1$ provability if each is included in the other up to $\mathbf{\Sigma}^1_1$ provability.
$$T\equiv^{\mathbf{\Sigma}^1_1} U \defiff \big(T\sqsubseteq^{\mathbf{\Sigma}^1_1} U \textrm{ and } U\sqsubseteq^{\mathbf{\Sigma}^1_1} T \big).$$
We note that over $\mathsf{ACA}_0$ the formula $T\equiv^{\mathbf{\Sigma}^1_1} U$ could be equivalently transformed to a $\Sigma^1_1$-formula.  

Recall that an $\omega$-model $\mathfrak{M}$ of second-order arithmetic is a structure whose interpretation of the natural numbers is standard, and the sort of sets of naturals is interpreted by some subset $\mathcal{S}_{\mathfrak{M}}$ of $\mathcal{P}(\mathbb{N})$. We reserve Fraktur letters $\mathfrak{M}$, $\mathfrak{N}$, ... for $\omega$-models. If $\mathcal{S}_{\mathfrak{M}}$ is countable, then the $\omega$-model $\mathfrak{M}$ is called countable. Formally, an $\omega$-model $\mathfrak{M}$ is a code for a countable set $\mathcal{S}_{\mathfrak{M}}$ of sets. Since we will formalize many results in $\mathsf{ACA}_0$, we must be careful in our treatment of $\omega$-models.

We recall the notion of a \emph{weak model} \cite[Definition~II.8.9]{simpson2009subsystems}. A weak model $\mathfrak{N}$ of a theory $K$ is a pair $\langle D_{\mathfrak{N}},\models_{\mathfrak{N}}\rangle$, where $D_{\mathfrak{N}}$ is the domain of the model and $\models_{\mathfrak{N}}$ is a partial satisfaction relation that covers all propositional combinations of atomic formulas and subformulas of axioms $K$. The partial satisfaction relation $\models_{\mathfrak{N}}$ should satisfy the usual compositionality conditions. Even over $\mathsf{RCA}_0$ the existence of a weak model of a theory $K$ implies the consistency of $K$ \cite[Definition~II.8.10]{simpson2009subsystems}. By a \emph{weak $\omega$-model} we mean, of course, a weak model $\mathfrak{N}$ whose domain $D_{\mathfrak{N}}$ is $\mathbb{N}$.

We formulate $\omega$-model reflection principles instead in terms of full satisfaction classes. Full satisfaction predicates are available only in $\mathsf{ACA}_0^+$, but partial satisfaction predicates are available in $\mathsf{ACA}_0$. Let $L_{\mathfrak{M}}$ be the set of all $\mathbf{L}_2$-sentences that contain constants $C_X$ only for $X\in\mathcal{S}_{\mathfrak{M}}$. A \emph{satisfaction class} is an assignment of truth-values to all formulas that satisfies the Tarski clauses. Provably in $\mathsf{ACA}_0^+$ (but not in $\mathsf{ACA}_0^+$) every $\omega$-model can be enriched with a full satisfaction class. Here is how we make sense of the notation ``$\mathfrak{M}\models\varphi$'' in $\mathsf{ACA}_0$: for an $\omega$-model $\mathfrak{M}$ and formula $\varphi$, $\mathfrak{M}\models \varphi$ if and only if $\varphi$ is true according to some satisfaction class $\mathcal{C}$ for $\mathfrak{M}$. We will formulate reflection principles in terms of full satisfaction classes. However, we will sometimes derive reflection principles concerning full satisfaction from principles concerning weak $\omega$-models with partial satisfaction classes.


The expression $\mathfrak{M}\models T$ for an $\mathbf{L}_2$-theory $T$ means that $T\subseteq L_{\mathfrak{M}}$ and for any axiom $\varphi$ of $T$ we have $\mathfrak{M}\models\varphi$. For notational convenience, when discussing $\omega$-models, instead of writing $X\dot\in\mathcal{S}_{\mathfrak{M}}$ we use the usual membership relation $X\in\mathfrak{M}$. We will also be interested in the reflection principle ``every set is contained in an $\omega$-model of $T$.'' This is equivalent to the claim that every true $\mathbf{\Sigma}^1_1$ (alternatively, $\mathbf{\Sigma}^1_2$) sentence is satisfied by some $\omega$-model of $T$, a fact which we will use in the proof of Theorem \ref{reduction_for_omega_models}.

For $n\ge 1$ we formulate the reflection principle $\mathbf{\Pi}^1_n\mbox{-}\omega\mathsf{RFN}(T)$ that says ``any $\mathbf{\Pi}^1_n$ sentence that holds in all $\omega$-models of $T$ is true.'' Formally, this is the sentence:
$$\forall\varphi\in\mathbf{\Pi}^1_n\big(\forall\mathfrak{M}(\mathfrak{M}\vDash T \rightarrow \mathfrak{M}\vDash\varphi)\rightarrow\mathsf{Tr}_{\mathbf{\Pi}^1_n}(\varphi)\big).$$
Contraposing the statement of $\mathbf{\Pi}^1_n\mbox{-}\omega\mathsf{RFN}(T)$ shows that it is equivalent (in $\mathsf{ACA}_0$) to the principle $\mathbf{\Sigma}^1_n$-$\omega\mathsf{Con}(T)$, which says ``for any true $\mathbf{\Sigma}^1_n$-sentence $\varphi$, there is an $\omega$-model of $T$ in which $\varphi$ holds.'' Formally this is the  sentence $$\forall \varphi\in \mathbf{\Sigma}^1_n\;\big(\mathsf{Tr}_{\mathbf{\Sigma}^1_n}(\varphi)\to \exists \mathfrak{M}\; \mathfrak{M}\models T+\varphi\big).$$

\begin{remark}
Note that, provably in $\mathsf{ACA}_0$, each principle of the form $\mathbf{\Pi}^1_n\text{-}\omega\mathsf{RFN}(T)$ implies $\mathsf{ACA}_0^+$ since we only consider theories $T$ that extend $\mathsf{ACA}_0$. Indeed, any of the reflection principles could be applied to the sentence $0\in C_A\lor \lnot 0\in C_A$. Thus for any set $A$ there exists a collection $S$ of sets containing $A$ and satisfying $\mathsf{ACA}_0$. Hence $S$ contains $A^{(n)}$, for every $n$, and thus we can construct $A^{(\omega)}$ from $S$.  
\end{remark}

Once again, since this is a consistency-like predicate we can formulate a corresponding provability-like predicate $$\mathbf{\Sigma}^1_n\mbox{-}\omega\mathsf{Prv}(T,\varphi):=\neg\mathbf{\Sigma}^1_n\mbox{-}\omega\mathsf{Con}(T+\neg\varphi).$$ A short argument demonstrates that $\mathbf{\Sigma}^1_n\mbox{-}\omega\mathsf{Prv}(T,\varphi)$ expresses that there exists a true $\mathbf{\Sigma}^1_n$-sentence $\psi$ such that $\varphi$ holds in all $\omega$-models of $T+\psi$:
\begin{flalign*}
\mathbf{\Sigma}^1_n\mbox{-}\omega\mathsf{Prv}(T,\varphi) &\equiv \lnot \mathbf{\Sigma}^1_n\mbox{-}\omega\mathsf{Con}(T+\lnot\varphi)\\
&\equiv \neg \forall \psi\in \mathbf{\Sigma}^1_n\;\big(\mathsf{Tr}_{\mathbf{\Sigma}^1_n}(\psi)\to \exists \mathfrak{M}\; (\mathfrak{M}\models T+\neg\varphi+\psi)\big)\\
&\equiv \exists \psi\in\mathbf{\Sigma}^1_n\big(\mathsf{Tr}_{\mathbf{\Sigma}^1_n}(\psi)\wedge \neg \exists \mathfrak{M}\;( \mathfrak{M}\models T+\neg\varphi+\psi)\big) \\
&\equiv \exists \psi\in\mathbf{\Sigma}^1_n\big(\mathsf{Tr}_{\mathbf{\Sigma}^1_n}(\psi)\wedge \forall \mathfrak{M}(\mathfrak{M}\models T+\psi \rightarrow \mathfrak{M}\models \varphi)\big)
\end{flalign*}
In \textsection \ref{equivalent_forms} we will see that over $\mathsf{ACA}_0$ the formula $\mathbf{\Sigma}^1_n\mbox{-}\omega\mathsf{Prv}(T,\varphi)$ is equivalent to the (appropriately formalized) fact that for some true $\mathbf{\Sigma}^1_n$ sentence $\psi$, there exists a cut-free $\omega$-proof of the sequent $\lnot T,\lnot \psi,\varphi$.



\subsection{Linear Orders and Iterated Reflection}\label{linear_orders}

We consider linear orders $\alpha$ defined on subsets of natural numbers. Formally $\alpha$ is a set encoding a pair $\langle D_{\alpha},\prec_{\alpha}\rangle$, where $D_\alpha\subseteq\mathbb{N}$ and $\prec_{\alpha}\subseteq D_{\alpha}^2$ is a strict linear order. For $x\in D_\alpha$ we denote by $\mathsf{cone}(\alpha,x)$ the set $\{y\in D_\alpha\mid y\prec_\alpha x\}$. Clearly, $\mathsf{cone}$ is a $\Pi^1_0$ comprehension term.

We will now turn to defining iterations of reflection principles along linear orders $\alpha$. Let us consider some reflection principle $\mathsf{RFN}(T)$ that is one of  $\mathbf{\Pi}^1_n\mbox{-}\mathsf{RFN}(T)$ or $\mathbf{\Pi}^1_n\mbox{-}\omega\mathsf{RFN}(T)$. In order to define $\mathbf{R}^{\cdot}(\cdot)$ we will define formulas $\mathsf{RFN}^{\mathsf{it}}(T,\alpha)$, where $T$ and $\alpha$ are free set variables and $x$ is a free number variable. 
We define $\mathsf{RFN}^{\mathsf{it}}(T,\alpha)$ as a fixed point that satisfies:
$$\mathsf{ACA}_0\vdash \mathsf{RFN}^{\mathsf{it}}(T,\alpha)\mathrel{\leftrightarrow} \mathsf{RFN}\Big(T+\big\{\mathsf{RFN}^{\mathsf{it}}\big(\dot T,\mathsf{cone}(\dot{\alpha},\dot{x})\big)\mid x\in D_\alpha\big\}\Big).$$
We put $$\mathbf{R}^\alpha(T)=T+\big\{\mathsf{RFN}^{\mathsf{it}}(T,\mathsf{cone}\big(\alpha,\dot{x})\big)\mid x\in D_\alpha\big\}.$$
And thus
$$\mathsf{ACA}_0\vdash \mathsf{RFN}^{\mathsf{it}}(T,\alpha)\mathrel{\leftrightarrow} \mathsf{RFN}\big(\mathbf{R}^{\alpha}(T)\big).$$
Clearly $\mathbf{R}^\alpha(T)$ is a $\Pi^1_0$-comprehension term depending on $\alpha$ and $T$.

We adopt the following notational conventions: If $\mathsf{RFN}(T)$ is  $\mathbf{\Pi}^1_n\mbox{-}\mathsf{RFN}(T)$, then $\mathbf{R}^\alpha(T)$ will be denoted $\mathbf{\Pi}^1_n\mbox{-}\mathbf{R}^\alpha(T)$ and if $\mathsf{RFN}(T)$ is  $\mathbf{\Pi}^1_n\mbox{-}\omega\mathsf{RFN}(T)$, then $\mathbf{R}^\alpha(T)$ will be denoted $\mathbf{\Pi}^1_n\mbox{-}\omega\mathbf{R}^\alpha(T)$.

A homomorphism of linear orders $f\colon \alpha\to \beta$ is a map $f\colon D_\alpha\to D_\beta$ such that $x\prec_{\alpha} y\Rightarrow f(x)\prec_{\beta} f(y)$, for any $x,y\in D_\alpha$.

\begin{lemma}[$\mathsf{ACA}_0$]\label{hom_lemma}
Suppose there is a homomorphism $f\colon \alpha\to \beta$. Then for any $\mathbf{L}_2$-theory $T$ we have $\mathbf{R}^{\alpha}(T)\sqsubseteq^{\mathbf{\Sigma}^1_1}\mathbf{R}^{\beta}(T)$.
\end{lemma}
\begin{proof}
We prove the lemma by L\"ob's theorem. That is we need to prove the lemma in $\mathsf{ACA}_0$ with additional assumption that $\mathsf{Prv}(\mathsf{ACA}_0,\Lambda)$, where $\Lambda$ statement of the lemma. 

We reason in $\mathsf{ACA}_0+\mathsf{Prv}(\mathsf{ACA}_0,\Lambda)$. Let us consider some homomorphism $f\colon \alpha\to \beta$ and an $\mathbf{L}_2$-theory $T$. We claim that $\mathbf{R}^{\alpha}(T)\sqsubseteq^{\mathbf{\Sigma}^1_1}\mathbf{R}^{\beta}(T)$. For this it is enough to show that any axiom of $\mathbf{R}^{\alpha}(T)$ is $\mathbf{\Sigma}^1_1$-provable in $\mathbf{R}^{\beta}(T)$. All axioms of $T$ are axioms of both $\mathbf{R}^{\alpha}(T)$ and $\mathbf{R}^{\beta}(T)$. Thus it is enough to show that any sentence $\mathsf{RFN}^{\mathsf{it}}(T,\mathsf{cone}(\dot{\alpha},\dot{x}))$ is $\mathbf{\Sigma}^1_1$-provable in $\mathbf{R}^{\beta}(T)$. Let us prove that $\mathsf{RFN}^{\mathsf{it}}(T,\mathsf{cone}(\dot{\alpha},\dot{x}))$ is $\mathbf{\Sigma}^1_1$-implied in $\mathsf{ACA}_0$  by $\mathsf{RFN}^{\mathsf{it}}(T,\mathsf{cone}(\dot{\beta},\dot f(\dot x)))$. By the construction of $\mathsf{RFN}^{\mathsf{it}}$ we need to show that $\mathsf{RFN}(\mathbf{R}^{\mathsf{cone}(\dot{\alpha},\dot{x})}(T))$ is $\mathbf{\Sigma}^1_1$-implied in $\mathsf{ACA}_0$  by $\mathsf{RFN}(\mathbf{R}^{\mathsf{cone}(\dot{\beta},\dot f(\dot x))}(T))$. Since $\mathsf{RFN}$ is at least as strong as $\mathbf{\Pi}^1_1\mbox{-}\mathsf{RFN}$, it is enough to show that $\Sigma^1_1$-provably in $\mathsf{ACA}_0$ we have
\begin{equation}\label{hom_lemma_eq1}\mathbf{R}^{\mathsf{cone}(\dot{\alpha},\dot{x})}(T)\sqsubseteq^{\mathbf{\Sigma}^1_1} \mathbf{R}^{\mathsf{cone}(\dot{\beta},\dot f(\dot x))}(T).\end{equation}
Let $H$ be the formula that says ``$\dot\alpha$ and $\dot\beta$ are linear orders and $\dot f$ is a homorphism $\dot f\colon \dot \alpha\to \dot \beta$.'' Note that $H$ is a true $\mathbf{\Pi}^1_0$ sentence. So we need to show that $\mathsf{ACA}_0+H$ proves (\ref{hom_lemma_eq1}). And since by our assumption $\mathsf{ACA}_0$ proves $\Lambda$ (recall that $\Lambda$ is the statement of the lemma), it is enough to show that $\mathsf{ACA}_0+H+\Lambda$ proves (\ref{hom_lemma_eq1}). The latter is simply a direct application of the lemma inside $\mathsf{ACA}_0$-provability to the homomorphism $\dot f{\upharpoonright} \mathsf{cone}(\dot\alpha,\dot x)\colon \mathsf{cone}(\dot\alpha,\dot x)\to \mathsf{cone}(\dot\beta,\dot f(\dot x))$.\end{proof}

 Lemma \ref{hom_lemma_eq1} implies that (provably in $\mathsf{ACA}_0$) up to $\mathbf{\Sigma}^1_1$ deductive equivalence the theories $\mathbf{R}^\alpha(T)$ depend only on  the order type of $\alpha$. Due to this we will not care about the particular numerical representations of the countable linear orders that we will consider. And we will freely switch between theories $\mathbf{R}^{\alpha}(T)$ and $\mathbf{R}^{\beta}(T)$ when $\alpha$ and $\beta$ are isomorphic.

 We identify natural numbers $n$ with the linear orders $(\{0,\ldots,n-1\},<)$. For linear orders $\alpha,\beta$, let us denote by $\alpha +\beta$ their ordered sum, i.e., the domain of $\alpha+\beta$ is the disjoint union $D_\alpha\sqcup D_\beta$ and any element of $\alpha$ is smaller than any element of $\beta$.

Using Lemma \ref{hom_lemma_eq1} it is easy to prove the following lemma.
\begin{lemma}[$\mathsf{ACA}_0$]\label{plus_one_lemma}
For any $\mathbf{L}_2$-theory $T$ and linear order $\alpha$:
$$\mathbf{R}^{\alpha+1}(T)\equiv^{\mathbf{\Sigma}^1_1} T+\mathsf{RFN}\big(\mathbf{R}^\alpha(T)\big).$$
\end{lemma}

We write $\alpha<^{\mathsf{cn}}\beta$, if there exists $x\in D_\beta$ and homomorphism $f\colon \alpha\to \mathsf{cone}(\beta,x)$.
Combining Lemmas \ref{hom_lemma} and \ref{plus_one_lemma} we see that:
\begin{corollary}[$\mathsf{ACA}_0$]\label{low-cones}
For any $\mathbf{L}_2$-theory $T$ and linear orders $\alpha<^{\mathsf{cn}}\beta$ we have
$$\mathsf{RFN}(\mathbf{R}^{\alpha}(T)) \sqsubseteq^{\mathbf{\Sigma}^1_1} \mathbf{R}^{\beta}(T).$$
\end{corollary}

\begin{lemma}[$\mathsf{ACA}_0$]
For any $\mathbf{L}_2$-theory $T$ and linear orders $\alpha,\beta$:
$$\mathbf{R}^{\alpha+\beta}(T)\equiv^{\mathbf{\Sigma}^1_1} \mathbf{R}^{\beta}\big(\mathbf{R}^\alpha(T)\big).$$
\end{lemma}
\begin{proof}
We prove the claim by L\"{o}b's Theorem. Namely, \textbf{we work in} $\mathsf{ACA}_0$ and suppose that $\mathsf{ACA}_0$ proves $\Lambda$, the statement of the lemma. It suffices to derive $\Lambda$.

We reason as follows, using our assumption that $\mathsf{ACA}_0$ proves $\Lambda$ to get the equivalence on the second line.
\begin{flalign*}
\mathbf{R}^{\alpha+\beta}(T) & \equiv \{\mathsf{RFN}\big(\mathbf{R}^{\alpha+\gamma}(T)\big) : \gamma<\beta\}\\
& \equiv^{\mathbf{\Sigma}^1_1} \{\mathsf{RFN}\big(\mathbf{R}^\gamma (\mathbf{R}^\alpha (T))\big): \gamma<\beta \}\\
&\equiv \mathbf{R}^\beta(\mathbf{R}^\alpha(T))
\end{flalign*}

This completes the proof of the lemma.
\end{proof}

\begin{lemma}[$\mathsf{ACA}_0$]
For any $\mathbf{L}_2$-theory $T$ and linear order $\alpha$ we have
$$\mathsf{RFN}(\mathbf{R}^{\alpha}(T))\;\Rightarrow\; \mathsf{WO}(\alpha).$$
\end{lemma}
\begin{proof}We prove the claim by L\"{o}b's Theorem. Namely, \textbf{we work in} $\mathsf{ACA}_0$ and suppose that $\mathsf{ACA}_0$ proves $\Lambda$, the statement of the lemma. It suffices to derive $\Lambda$.

Let $T$ be an $\mathbf{L}_2$-theory and $\alpha$ a linear order. Note that since $\mathsf{ACA}_0$ proves $\Lambda$ so does $\mathbf{R}^\alpha(T).$ Thus,
$$ \forall \beta<\alpha \Big( \mathbf{R}^\alpha(T) \vdash \mathsf{RFN}(\mathbf{R}^{\beta}(T)) \rightarrow \mathsf{WO}(\beta) \Big). $$

On the other hand, by the definition $\mathbf{R}^\alpha$,
$$\forall \beta<\alpha \Big( \mathbf{R}^\alpha(T) \vdash \mathsf{RFN}(\mathbf{R}^\beta(T)) \Big).$$

Combining these two observations, for each $\beta<\alpha$, $\mathbf{R}^\alpha(T) \vdash \mathsf{WO}(\beta)$. Thus,
$$\mathsf{RFN}(\mathbf{R}^{\alpha}(T))\;\Rightarrow\; \mathsf{WO}(\alpha).$$
This completes the proof of the lemma.\end{proof}

\begin{corollary}[$\mathsf{ACA}_0$]\label{inconsistency}
For any $\mathbf{L}_2$-theory $T$ and ill-founded linear order $\alpha$ we have
$$\mathbf{R}^{\alpha}(T)\equiv^{\mathbf{\Sigma}^1_1}\bot.$$
\end{corollary}

Corollary \ref{inconsistency} can also be inferred from Theorem 3.2 in \cite{pakhomov2018reflection}, which says that iterating $\mathsf{RFN}_{\Pi^1_1}$ (or any stronger reflection principle) along an ill-founded order produces $\Pi^1_1$ unsound theories; such theories clearly $\Sigma^1_1$ prove $\bot$.

For reflection schemes in first-order arithmetic there is a partial conservation result known as the \emph{reduction property} \cite[Lemma~4.2]{beklemishev2003proof}. Lemma \ref{reduction_property} below is a variant of the reduction property for the reflection principles considered in this paper. Our proof of Lemma \ref{reduction_property} differs significantly from Beklemishev's proof of \cite[Lemma~4.2]{beklemishev2003proof} (we use model-theoretic methods, wheres Beklemishev's proof is completely syntactic). We use semantic methods because we find them more convenient here; our choice does not indicate any fundamental difference between the two reduction properties. In fact, the methods from \cite{beklemishev2003proof} could be used to prove our reduction property and our method could be used to prove the reduction property from \cite{beklemishev2003proof}. We note that our proof of Lemma \ref{reduction_property} is based on ideas from \cite{avigad2002saturated}. 


\begin{definition}
For a syntactic complexity $\Gamma$, we write $T\sqsubseteq_\Gamma U$ if every $\Gamma$ consequence of $T$ is a consequence of $U$, i.e., for every $\varphi\in\Gamma$, $T\vdash \varphi$ implies $U\vdash\varphi$.

We write $T\equiv_\Gamma U$ if both $T\sqsubseteq_\Gamma U$ and  $U\sqsubseteq_\Gamma T$.

Note that we write $T\sqsubseteq^{\mathbf{\Sigma^1_1}}_\Gamma U$ if: $$\forall \varphi \in \Gamma\Big( \mathbf{\Sigma^1_1}\text{-}\mathsf{Prv}(T,\varphi)\to\mathbf{\Sigma^1_1}\text{-}\mathsf{Prv}(U,\varphi)\Big).$$
Finally, we write $T\equiv^{\mathbf{\Sigma^1_1}}_\Gamma U$ if $T\sqsubseteq^{\mathbf{\Sigma^1_1}}_\Gamma U$ and $U\sqsubseteq^{\mathbf{\Sigma^1_1}}_\Gamma T$.
\end{definition}

\begin{definition}\label{countable_fragmet}
A \emph{countable fragment} $H$ of $\mathbf{L}_2$ is a countable set of $\mathbf{L}_2$-formulas such that for any $L_2$ formula $\varphi(X_1,...,X_n)$ and set constants $C_1,...,C_n$ occurring in $H$ the formula $\varphi(C_1,...,C_n)$ is in $H$. 
\end{definition}

\begin{lemma}[$\mathsf{ACA}_0$]\label{rule_cons}
Let $T$ be some $\mathbf{\Pi}^1_{n+1}$-axiomatizable $\mathbf{L}_2$-theory and let $H$ be a countable fragment containing $T$. Let $U$ be an extension of $T$ by a series of axioms indexed by $i\in I$ of the form $\forall \vec{x}_i,\vec{X}_i\big(\varphi_i(\vec{x}_i,\vec{X}_i)\to \psi_i(\vec{x}_i,\vec{X}_i)\big)$, where $\varphi_i(\vec{x}_i,\vec{X}_i)\in \mathbf{\Pi}^1_{n}\cap H$ and $\psi_i(\vec{x}_i,\vec{X}_i)\in \mathbf{\Pi}^1_{n+1}\cap H$. Finally let $V$ be the closure of $T$ under all the rules $$\frac{\forall \vec{x}_i,\vec{X}_i,\vec{y},\vec{Y}\big(\theta(\vec{x}_i,\vec{X}_i,\vec{y},\vec{Y})\lor \varphi_i(\vec{x}_i,\vec{X}_i)\big)}{\forall \vec{x}_i,\vec{X}_i,\vec{y},\vec{Y}\big(\theta(\vec{x}_i,\vec{X}_i,\vec{y},\vec{Y})\lor \psi_i(\vec{x}_i,\vec{X}_i)\big)},\text{ where $i\in I$ and $\theta\in\mathbf{\Pi}^1_n\cap H$.}$$
Then  $U$ is a $\mathbf{\Pi}^1_n\cap H$-conservative extension of $V$.
\end{lemma}
\begin{proof}
It is trivial to see that $U$ is closed under all the rules that we added to $V$ and hence $U$ is indeed an extension of $V$. For the rest of the proof we fix a $\mathbf{\Pi}^1_n\cap H$-sentence $\xi$ that isn't provable in $V$ and construct a model of $U+\lnot \xi$.

 Let $H^+$ be the extension of $H$ by countable families of set and natural number constants: $C_0,C_1,\ldots$ and $c_0,c_1,\ldots$. We denote as $\Pi^{1+}_m$ and $\Sigma^{1+}_m$ the classes of $\Pi^1_m$ and $\Sigma^1_m$ formulas where we additionally allow all constants from $H^+$. We define a sequence $V_0\subseteq V_1\subseteq V_2\subseteq \ldots$ of theories of the language $H^+$ such that each of them extends $V$ by finitely many $\dot{\Sigma^1_m}$ sentences.  We put $V_0=V+\lnot \xi$.  We fix an enumeration $\eta_0(x,X),\eta_1(x,X),\ldots$ of all $\Sigma^{1+}_{n}$ formulas without other free variables. If $V_{i}+\exists x,X \eta_i(x,X)$ is inconsistent, then we put $V_{i+1}=V_i$. Otherwise we choose constants $c_k,C_l$ that don't appear among the axioms of $V_i$ and put $V_{i+1}=U_i+\eta_i(c_k,C_l)$. We put $V_\omega=\bigcup_{i\in \omega} U_i$.

Now consider the $H^+$ model $\mathfrak{M}$ whose domain consists of all terms from $H^+$ and such that, for all atomic formulas $\chi$, we have $\mathfrak{M}\models \chi$ iff $V_\omega\vdash \chi$. We claim that $\mathfrak{M}\models U+\lnot \theta$. By an easy induction on construction of $\Sigma^{1+}_{n}$-formulas $\nu$ we show that $\mathfrak{M}\models \nu$ iff $V_\omega\vdash \nu$. Further this implies that for $\Pi^{1+}_{n+1}$-formulas $\nu$ if $V_\omega\vdash \nu$, then $\mathfrak{M}\models \nu$. 
Hence $\mathfrak{M}\models T+\lnot \theta$. Also, observe that for any set constant $A$ from $H^+$ there clearly is some $C_i$ such that  $V_\omega\vdash A=C_i$ (since $V\vdash \exists X,x  (A=X)$ and thus $\exists X,x  (A=X)$ is consistent with any $V_i$) and for any natural number term $t$ there is some $c_i$ such that $T_i\vdash c_i=t$. 

To finish the proof we only need to show that $\mathfrak{M}\models \forall \vec{X}_i,\vec{x}_i \big(\varphi_i(\vec{X}_i,\vec{x}_i)\to \psi_i (\vec{X}_i,\vec{x}_i)\big)$, for all $i\in I$. We fix   $i\in I$ and check that $\mathfrak{M}\models \forall \vec{X}_i,\vec{x}_i\big(\varphi_i(\vec{X}_i,\vec{x}_i)\to\psi_i(\vec{X}_i,\vec{x}_i)\big)$.  For this we fix  vectors $\vec{a},\vec{A}$ of constants from the alphabets $c_0,c_1,\ldots$ and  $C_0,C_1,\ldots$ that have the same dimensions as $\vec{x}_i,\vec{X}_i$ such that $\mathfrak{M}\models \varphi_i(\vec{a},\vec{A})$. We now just need to check that $\mathfrak{M}\models \psi_i(\vec{a},\vec{A})$. To verify the latter, we check that $V_\omega\vdash \psi_i(\vec{a},\vec{A})$. Indeed, $V_\omega\vdash \varphi_i(\vec{a},\vec{A})$. Hence $T_i\vdash \varphi_i(\vec{a},\vec{A})$, for some $i$. The theory $V_i$ is of the form $V+\eta(\vec{a},\vec{A},\vec{b},\vec{B})$, where $\eta(\vec{x}_i,\vec{X}_i,\vec{y},\vec{Y})$ is an $\mathbf{\Sigma}^1_n\cap H$ formula and $\vec{b},\vec{B}$ are some constants from the alphabets $C_0,C_1,\ldots$ and $c_0,c_1,\ldots$ such that $\vec{b},\vec{B}$ are pairwise distinct from $\vec{a},\vec{A}$. Thus $$V\vdash \lnot \eta(\vec{a},\vec{A},\vec{b},\vec{B})\lor \varphi_i(\vec{A},\vec{a}).$$ Since the language $H$ of $V$ doesn't have the constants $\vec{b},\vec{B},\vec{a},\vec{A}$ we could treat them as free variables and hence $$V\vdash \forall \vec{x}_i,\vec{X}_i,\vec{y},\vec{Y} \big(\lnot \eta(\vec{x}_i,\vec{X}_i,\vec{y},\vec{Y})\lor \varphi_i(\vec{x}_i,\vec{X}_i)\big).$$ Using the closure of $V$ under the rule 
$$\frac{\forall \vec{x}_i,\vec{X}_i,\vec{y},\vec{Y}\big(\lnot \eta(\vec{x}_i,\vec{X}_i,\vec{y},\vec{Y})\lor \varphi_i(\vec{x}_i,\vec{X}_i)\big)}{\forall \vec{x}_i,\vec{X}_i,\vec{y},\vec{Y}\big(\lnot \eta(\vec{x}_i,\vec{X}_i,\vec{y},\vec{Y})\lor \psi_i(\vec{x}_i,\vec{X}_i)\big)}$$
we conclude that $$V\vdash \forall \vec{x}_i,\vec{X}_i,\vec{y},\vec{Y}\big(\lnot \eta(\vec{x}_i,\vec{X}_i,\vec{y},\vec{Y})\lor \psi_i(\vec{x}_i,\vec{X}_i)\big).$$ Since $V_\omega\vdash \eta(\vec{a},\vec{A},\vec{b},\vec{B})$ we conclude that  $V_\omega\vdash \psi_i(\vec{A},\vec{a})$.\end{proof}

\begin{lemma}[$\mathsf{ACA}_0$]\label{reduction_property} Suppose $H$ is a countable fragment of $\mathbf{L}_2$, $n\ge 1$ is a natural number,  $T$ is  $\mathbf{\Pi}^1_n \cap H$-axiomatized theory, and a theory $U$ is such that $C_U$ is an $H$-constant. Then 
\begin{equation} \label{reduction_i1}T+\frac{\varphi\in \mathbf{\Pi}^1_n \cap H}{\mathbf{\Pi}^1_{n}\mbox{-}\mathsf{RFN}(U+\varphi)}\equiv_{\mathbf{\Pi}^1_n} T+\mathbf{\Pi}^1_{n+1}\mbox{-}\mathsf{RFN}(U),\end{equation}
And if $n\ge 2$ then
\begin{equation} \label{reduction_i2}T+\frac{\varphi\in\mathbf{\Pi}^1_n \cap H}{\mathbf{\Pi}^1_{n}\mbox{-}\omega\mathsf{RFN}(U+\varphi)}\equiv_{\mathbf{\Pi}^1_n} T+\mathbf{\Pi}^1_{n+1}\mbox{-}\omega\mathsf{RFN}(U).\end{equation}
\end{lemma}
\begin{proof}
Since the proofs of both (\ref{reduction_i2}) and (\ref{reduction_i1}) are identical (modulo the switch of $\Pi^1_m\text{-}\omega\mathsf{RFN}$ with $\Pi^1_m\text{-}\mathsf{RFN}$), we will just cover the case of (\ref{reduction_i2}).

Observe that the schemata $\mathbf{\Pi}^1_{n+1}\textsf{-}\omega\textsf{RFN}(U)$ is equivalent to the schemata
$$\forall \vec{x},\vec{X}(\varphi(\vec{x},\vec{X})\to \mathbf{\Pi}^1_{n}\textsf{-}\omega\textsf{RFN}(U+\varphi(\dot{\vec{x}},\dot{\vec{X}})))\textsf{, where}\varphi\in \mathbf{\Pi}^1_n\cap H.$$
Thus Lemma \ref{rule_cons} it is enough to show that the theory $$U=T+\frac{\varphi\in\mathbf{\Pi}^1_n \cap H}{\mathbf{\Pi}^1_{n}\mbox{-}\omega\mathsf{RFN}(U+\varphi)}$$ is closed under the rule \begin{equation}\label{reduction_property_r1}\frac{\forall \vec{x},\vec{X},\vec{y},\vec{Y}(\theta(\vec{x},\vec{X},\vec{y},\vec{Y})\lor \varphi(\vec{x},\vec{X}))}{\forall \vec{x},\vec{X},\vec{y},\vec{Y}(\theta(\vec{x},\vec{X},\vec{y},\vec{Y})\lor  \mathbf{\Pi}^1_{n}\textsf{-RFN}(U+\varphi(\dot{\vec{x}},\dot{\vec{X}})))}\text{, where $\varphi,\theta\in \mathbf{\Pi}^1_n\cap H$}.\end{equation}
Indeed, assume $$U\vdash \forall \vec{x},\vec{X},\vec{y},\vec{Y}(\theta(\vec{x},\vec{X},\vec{y},\vec{Y})\lor \varphi(\vec{x},\vec{X}))\text{, where $\varphi,\theta\in \mathbf{\Pi}^1_n\cap H$}.$$ Then $$U\vdash \mathbf{\Pi}^1_n\mbox{-}\omega\mathsf{RFN}(U+\forall \vec{x},\vec{X},\vec{y},\vec{Y}(\lnot \theta(\vec{x},\vec{X},\vec{y},\vec{Y})\to \varphi(\vec{x},\vec{X}))).$$ Hence $$U\vdash \forall \vec{x},\vec{X},\vec{y},\vec{Y}(\mathbf{\Pi}^1_n\mbox{-}\omega\mathsf{RFN}(U+\lnot \theta(\dot{\vec{x}},\dot{\vec{X}},\dot{\vec{y}},\dot{\vec{Y}})\to \varphi(\dot{\vec{x}},\dot{\vec{X}}))).$$ Therefore, since $\lnot \theta$ is $\mathbf{\Sigma}^1_n$, we have $$U\vdash\forall \vec{x},\vec{X},\vec{y},\vec{Y}(\lnot \theta(\vec{x},\vec{X},\vec{y},\vec{Y})\to \mathbf{\Pi}^1_n\mbox{-}\omega\mathsf{RFN}(U+ \varphi(\dot{\vec{x}},\dot{\vec{X}}))).$$ Which verifies the closuer of $U$ under the rule (\ref{reduction_property_r1}).

\end{proof}


The following theorem, also provable using L\"{o}b's Theorem, expresses the systematic connection between $\omega$-model reflection and iterations of the Turing jump.

\begin{theorem}[$\mathsf{ACA}_0$] \label{omegaRFN_Turing}For any linear order $\alpha$ the following assertions are equivalent:
\begin{enumerate}
    \item $\mathbf{\Pi}^1_2\mbox{-}\omega\mathsf{RFN}\big(\mathbf{\Pi}^1_2\mbox{-}\omega\mathbf{R}^{\alpha}(\mathsf{ACA}_0)\big)$;
    \item For any set $X$, its $\omega^{1+\alpha}$-th Turing jump $X^{(\omega^{1+\alpha})}$ exists.
\end{enumerate}
\end{theorem}

\begin{proof}
We prove both directions using L\"{o}b's Theorem.

For the first direction, we work in $\mathsf{ACA}_0$ and assume that $\mathsf{ACA}_0$ proves $(1) \Rightarrow (2)$. We want to show that $(1)\Rightarrow (2)$.

So assume (1). We know from (1) that there is an $\omega$-model $\mathfrak{M}$ of $\mathbf{\Pi}^1_2\mbox{-}\omega\mathbf{R}^{\alpha}(\mathsf{ACA}_0)$ containing $X$. Note that $\mathfrak{M} \models \mathbf{\Pi}^1_2\mbox{-}\omega\mathsf{RFN}(\mathbf{\Pi}^1_2\mbox{-}\omega\mathbf{R}^{\beta}(\mathsf{ACA}_0))$ for each $\beta<\alpha$.  Since $\mathsf{ACA}_0$ proves $(1)\Rightarrow (2)$, $\mathfrak{M} \models$ ``$X^{(\omega^{1+\beta})}$ exists'' for each $\beta<\alpha$. And since $\mathfrak{M}$ is an $\omega$-model, $\mathfrak{M}$ correctly identifies $X^{(\omega^{1+\beta})}$ for each $\beta<\alpha$. We use arithmetical comprehension to extract $X^{(\omega^{1+\alpha})}$ from $\mathfrak{M}$.

For the second direction, we work in $\mathsf{ACA}_0$ and assume that $\mathsf{ACA}_0$ proves $(2) \Rightarrow (1)$. We want to show that $(2)\Rightarrow (1)$.

So assume (2). Note that (1) is equivalent to $\mathbf{\Sigma}^1_2$-$\omega\mathsf{Con}( \mathbf{\Pi}^1_2\mbox{-}\omega\mathbf{R}^{\alpha}(\mathsf{ACA}_0)  )$, which says ``for any true $\mathbf{\Sigma}^1_2$ sentence $\varphi$, there is an $\omega$-model of $\mathbf{\Pi}^1_2\mbox{-}\omega\mathbf{R}^{\alpha}(\mathsf{ACA}_0)$ in which $\varphi$ holds.'' So let $\varphi$ be a true $\mathbf{\Sigma}^1_2$ sentence and let $X$ be its witness. From (2), we know that $X^{(\omega^{1+\alpha})}$ exists. We use this set to define an $\omega$-model $\mathfrak{M}$ containing $X$ and closed under the $\omega^{1+\beta}$ jump for all $\beta<\alpha$. Since $\mathsf{ACA}_0$ proves $(2) \Rightarrow (1)$, $\mathfrak{M}$ is a model of $\mathbf{\Pi}^1_2\text{-}\omega\mathsf{RFN}(\mathbf{\Pi}^1_2\text{-}\omega\mathbf{R}^\beta(\mathsf{ACA}_0)$ for all $\beta< \alpha$. So $\mathfrak{M}$ is an $\omega$-model of 
$\mathbf{\Pi}^1_2\text{-}\omega\mathbf{R}^\alpha(\mathsf{ACA}_0)$ in which $\varphi$ holds.
\end{proof}

Immediately from Theorem \ref{omegaRFN_Turing} we get

\begin{corollary}\label{reflection_ATR_0}
 $$\mathsf{ATR}_0\mathrel{\equiv} \forall \alpha \Big(\mathsf{WO}(\alpha) \rightarrow \mathbf{\Pi}^1_2\textrm{-}\omega\mathsf{RFN}\big(\mathbf{\Pi}^1_2\textrm{-}\omega\mathbf{R}^\alpha(\mathsf{ACA}_0)\big)\Big).$$
\end{corollary}



\section{Equivalent forms of $\omega$-model reflection}\label{equivalent_forms}
In this section we will show that with our choice of $\mathsf{ACA}_0$ as base system, the principle of $\omega$-model reflection is fairly robust with respect to the choice of particular formalization. Namely, we will show the equivalence of the variants of reflection based on $\omega$-models, $\omega$-proofs with cuts, and cut-free $\omega$-proofs. Note that David Fern\'{a}ndez-Duque \cite{fernndezduque2015impredicative} proved that for certain other similar reflection principles these equivalences aren't provable in $\mathsf{ACA}_0$; namely he considered reflection principles based on certain formalizations of provability in $\omega$-logic that were not based on the notion of $\omega$-proof. 

\subsection{Defining $\omega$-proofs for $\mathbf{L}_2$}

First let us formulate the variant of $\omega$-logic for $\mathbf{L}_2$. This logic will be a variant of the Tait calculus.  Formulas are built up from literals using the connectives $\land,\lor$ and quantifiers $\forall x$, $\exists x$, $\forall X$, $\exists X$. Literals are atomic $\mathbf{L}_2$-formulas $\varphi$ and their negations ${\sim}\varphi$. As usual for any formula $\varphi$, its negation $\lnot \varphi$ is the result of switching any connective and quantifier with the dual, switching positive literals $\varphi$ with ${\sim}\varphi$ and switching negative literals ${\sim}\varphi$ with $\varphi$. Sequents are at most countable sets of formulas without free natural number variables (we allow free set variables). The axioms and rules of the logic are:
\begin{center}
  \AxiomC{\phantom{$\Gamma$}}
  \RightLabel{, if $\mathsf{val}(t)=\mathsf{val}(v)$ ($\mathsf{Ax}_1$);}
  \UnaryInfC{$\Gamma,t=v$}
  \DisplayProof
  \hspace{4pt}
  \AxiomC{\phantom{$\Gamma$}}
  \RightLabel{, if $\mathsf{val}(t)\ne \mathsf{val}(v)$ ($\mathsf{Ax}_2$);}
  \UnaryInfC{$\Gamma,{\sim}t=v$;}
  \DisplayProof
  \hspace{4pt}
  \AxiomC{\phantom{$\Gamma$}}
  \RightLabel{, if $\mathsf{val}(t)\in A$ ($\mathsf{Ax}_3$);}
  \UnaryInfC{$\Gamma,t\in C_A$}
  \DisplayProof
  \hspace{4pt}
  \AxiomC{\phantom{$\Gamma$}}
  \RightLabel{, if $\mathsf{val}(t)\not\in A$ ($\mathsf{Ax}_4$);}
  \UnaryInfC{$\Gamma,{\sim}t\in C_A$}
  \DisplayProof
  \hspace{4pt}
  \AxiomC{\phantom{$\Gamma$}}
  \RightLabel{($\mathsf{Ax}_5$);}
  \UnaryInfC{$\Gamma,t\in X,{\sim}t\in X$}
  \DisplayProof
  \hspace{4pt}
  \AxiomC{$\Gamma,\varphi$}
  \AxiomC{$\Gamma,\psi$}
  \RightLabel{($\land\mbox{-}\mathsf{Int}$);}
  \BinaryInfC{$\Gamma,\varphi\land\psi$}
  \DisplayProof
  \hspace{4pt}
  \AxiomC{$\Gamma,\varphi,\psi$}
  \RightLabel{($\lor\mbox{-}\mathsf{Int}$);}
  \UnaryInfC{$\Gamma,\varphi\lor\psi$}
  \DisplayProof
  \hspace{4pt}
  \AxiomC{$\Gamma,\varphi(\underline{n})$,}
  \AxiomC{ for all $n\in\mathbb{N}$}
  \RightLabel{($\forall_1\mbox{-}\mathsf{Int}$);}
  \BinaryInfC{$\Gamma,\forall x\;\varphi(x)$}
  \DisplayProof
  \hspace{4pt}
  \AxiomC{$\Gamma,\varphi(t)$}
  \RightLabel{($\exists_1\mbox{-}\mathsf{Int}$);}
  \UnaryInfC{$\Gamma,\exists x\;\varphi(x)$}
  \DisplayProof
  \hspace{4pt}
  \AxiomC{$\Gamma,\varphi(Y)$}
  \RightLabel{, if $Y\not\in\mathsf{FV}(\Gamma)$($\forall_2\mbox{-}\mathsf{Int}$);}
  \UnaryInfC{$\Gamma,\forall X\;\varphi(X)$}
  \DisplayProof
  \hspace{4pt}
  \AxiomC{$\Gamma,\varphi(C_A)$}
  \RightLabel{($\exists_2\mbox{-}\mathsf{Int}_1$);}
  \UnaryInfC{$\Gamma,\exists X\;\varphi(X)$}
  \DisplayProof
  \hspace{4pt}
  \AxiomC{$\Gamma,\varphi(Y)$}
  \RightLabel{($\exists_2\mbox{-}\mathsf{Int}_2$);}
  \UnaryInfC{$\Gamma,\exists X\;\varphi(X)$}
  \DisplayProof
  \hspace{4pt}
    \AxiomC{$\Gamma,\varphi$}
    \AxiomC{$\Gamma,\lnot \varphi$}
  \RightLabel{($\mathsf{Cut}$).}
  \BinaryInfC{$\Gamma$}
    \DisplayProof
  \hspace{4pt}
  \AxiomC{$\Gamma$}
  \RightLabel{($\mathsf{Rep}$);}
  \UnaryInfC{$\Gamma$}
  \DisplayProof
\end{center}
A \emph{pre-proof} is any tree that accords with these axioms and rules in the sense that its leaves are axioms and each child node follows from applying one of the rules. Note that a pre-proof may be ill-founded. By a \emph{proof} we mean a well-founded pre-proof. A sequent $\Gamma$ is $\omega$-provable if there is a well-founded proof-tree with $\Gamma$ as its conclusion. We write $\vdash_{\omega}\Gamma$ if the sequent $\Gamma$ has an $\omega$-proof. And we write $\vdash_0\Gamma$ if the sequent $\Gamma$ has a cut-free $\omega$-proof.

\subsubsection{Details of encoding $\omega$-proofs}

We now describe in some detail how we encode infinitary proof trees in $\mathsf{ACA}_0$. We encode sequents as codes for countable sets of $\mathbf{L}_2$-formulas. Due to the way our encoding works, the same sequent could have multiple representations. Note that equality on codes of sequents coincides with extensional equality: $$X\dot=Y\defiff \forall Z(Z\dot\in X\mathrel{\leftrightarrow} Z\dot\in Y).$$
And it is easy to see that $X\dot=Y$ is equivalent to a $\Pi^1_0$ formula.

It is useful to define not only the notion of \emph{proof} but also the notion of \emph{pre-proof}, where a pre-proof is a possibly ill-founded derivation tree. More formally, a pre-proof $P$ is (a code for) a triple $\mathit{Sh}_P$, $\mathit{Sq}_P$, $\mathit{Rl}_P$.  Here $\mathit{Sh}_P$ is a ``proof-shape'' tree $\langle I_P,r_P,\prec_P\rangle$, where  $I_P\subseteq\mathbb{N}$ is the domain of the tree, $r_P\in I_P$ is the root of the tree, and $x \prec_P y$ is the binary relation on $I_P$ with the intended meaning that $x$ is a child of $y$. We require  that for any $i\in I_P$ there exists unique $\prec_P$-path from it to the root $$i=i_0\prec_P i_1\prec_P\ldots\prec_P i_n=r_P.$$ We require $\mathit{Sq}_P$ to be an assignment of sequents $\langle \Delta_i\mid i\in I_P\rangle$ to the nodes of the tree $\mathit{Sh}_P$. Finally, $\mathit{Rl}_P$ is an assignment of rules $\langle R_i\mid i\in I_P\rangle$ to the nodes of the tree $\mathit{Sh}_P$. Each $R_i$ contains all the information about the applied rule. First it contains the rule type ($\mathsf{Ax}_1$, $\mathsf{Ax}_2$, $\mathsf{Ax}_3$, $\mathsf{Ax}_4$, $\mathsf{Ax}_5$, $\land\mbox{-}\mathsf{Int}$, $\lor\mbox{-}\mathsf{Int}$, $\forall_1\mbox{-}\mathsf{Int}$, $\exists_1\mbox{-}\mathsf{Int}$, $\forall_2\mbox{-}\mathsf{Int}$, $\exists_2\mbox{-}\mathsf{Int}_1$,  $\exists_2\mbox{-}\mathsf{Int}_2$). And it contains the information specific to each particular rule type. Let us specify what this information is in the case when $R_i$ is of the type $\forall_1\mbox{-}\mathsf{Int}$, the cases of all the other rule types are analogous. The rule $R_i$ should be supplied with the sequent $\Gamma_{i}$, variable $x_{i}$, formula $\varphi_i(x_i)$ and sequents of indices of the premises $\langle p_{i,n}\mid n\in \mathbb{N}\rangle$. It is required that $\Delta_i\dot=(\Gamma_i,\forall x_i\;\varphi_i(x_i))$, that all $\Delta_{p_{i,n}}\dot=(\Gamma_i,\varphi_i(\underline{n}))$, that $\{p_{i,n}\mid n\in\mathbb{N}\}=\{j\in I_P\mid j\prec_P i\}$, and that $p_{i,n}$ are pairwise distinct. For a pre-proof $P$ the sequent $\Gamma_{r_P}$ is called the conclusion of $P$.  A pre-proof $P$ is called a proof if $\prec_P$ is a well-founded relation.

We write $\vdash_{\omega}\Gamma$ if the sequent $\Gamma$ has an $\omega$-proof. And we write $\vdash_0\Gamma$ if the sequent $\Gamma$ has a cut-free $\omega$-proof. We note that $\mathsf{ACA}_0$ cannot prove the full cut-elimination theorem for $\omega$-logic (cut-elimination for $\omega$-logic requires the system $\mathsf{ACA}_0^+$; however, $\mathsf{ACA}_0$ \emph{can} show that it is possible to eliminate all the cuts of the highest rank, see \cite[Theorem~6.4.1]{girard1987proof}). Due to this issue we formulate several variants of $\omega$-completeness theorems.

Recall that we write $\mathfrak{M}\models T$ if all axioms of theory $T$ hold in the model $\mathfrak{M}$.  At the same time for closed sequents $\Gamma$ (i.e. sequents without free variables) we will write $\mathfrak{M}\models \Gamma$ if some formula $\varphi\in \Gamma$ holds in $\mathfrak{M}$. This is an abuse of notation since both sequents and theories are represented by codes of (countable) sets of $\mathbf{L}_2$-formulas. However, it will be always clear from context whether a particular object is a theory or a sequent (in particular we denote theories by capital Latin letters $T,U$ and sequents by capital Greek letters $\Gamma,\Delta$). For a theory $T$ we denote by $\lnot T$ the sequent $\{\lnot \varphi \mid \varphi\mbox{ is an axiom of } T\}$.

\subsection{Completeness theorems for cut-free $\omega$-proofs}


We now describe in detail a completeness theorem for $\omega$-proofs with respect to $\omega$-models. Our completeness theorem is proved using Sch\"{u}tte's method of deduction chains. Thus, before proving the theorem we will work up to the definition of a \emph{deduction chain} for a sequent $\Gamma$ and a countable fragment $H$ of $\mathbf{L}_2$. First, the definition of a \emph{countable fragment} of $\mathbf{L}_2$:

\begin{remark}
When we are working with a sequent $\Gamma$ and a countable fragment $H$ (see Definition \ref{countable_fragmet}), we will assume that:
\begin{enumerate}
\item  $H$ comes with a fixed enumeration $Y_0,Y_1,...$ of the free set variables in $H$ that do not occur free in $\Gamma$.
    \item $H$ comes with a fixed enumeration $A_0,A_1,...$ where each $A_i$ is either:
\begin{enumerate}
    \item an $H$ formula $\varphi_i$ that does not start with $\exists$
    \item a pair $\langle \exists x \varphi_i(x),t\rangle$ where $\exists x \varphi_i(x)$ is an $H$ formula and $t$ is a closed term or
    \item a pair $\langle \exists X \varphi_i(X),U\rangle$ where $ \exists X \varphi_i(X)$ is an $H$ formula and $U$ is either a second order variable or second order constant.
\end{enumerate}
\end{enumerate}
We require that the sequence $A_0,A_1,...$ covers all formulas and pairs of the form we describe; moreover, we require that each such formula and pair occur infinitely many times in the enumeration.
\end{remark}





\begin{definition}
A sequent $\Delta$ is \emph{axiomatic} if it contains an instance of one of the axioms (1)--(5). 
\end{definition}


\begin{definition}
A \emph{deduction chain} for a sequent $\Gamma$ and a countable fragment of $H$ of $\mathbf{L}_2$ is a finite sequence $\Delta_0,\Delta_1,...,\Delta_k$ of sequents (i.e., countable sets) of constant $\mathbf{L}_2$ formulas satisfying the following conditions:
\begin{enumerate}
\item $\Delta_0$ is the sequent $\Gamma$.
\item For all numbers $i$ less than $k$, $\Delta_i$ is not axiomatic.
\item If $A_i$ is $\varphi\wedge\psi$ and $\varphi\wedge\psi \in \Delta_i$, then $\Delta_{i+1}$ is either $\Delta_i,\varphi$ or $\Delta_i,\psi$.
\item If $A_i$ is $\varphi\vee\psi$ and $\varphi\vee\psi\in\Delta_i$, then $\Delta_{i+1}$ is $\Delta_i,\varphi,\psi$.
\item If $A_i$ is $\forall x\varphi(x)$ and $\forall x\varphi(x)\in\Delta_i$, then, for some $n\in\mathbb{N}$, $\Delta_{i+1}$ is $\Delta_i,\varphi(\bar{n})$.
\item If $A_i$ is $\forall X \varphi(X)$ and $\forall X \varphi(X) \in \Delta_i$, then $\Delta_{i+1}$ is $\Delta_i,\varphi(Y_i)$.
\item If $A_i$ is $\langle \exists x \varphi(x),t\rangle$ and $\exists x \varphi(x) \in \Delta_i$, then $\Delta_{i+1}$ is $\Delta_i,\varphi(t)$.
\item If $A_i$ is $\langle \exists X \varphi(X),U\rangle$ and $\exists X \varphi(X)\in\Delta_i$, then $\Delta_{i+1}$ is $\Delta_i,\varphi(U)$.
\item Otherwise, $\Delta_{i+1} = \Delta_i$.
\end{enumerate}
This concludes the definition of \emph{deduction chains}.
\end{definition}

\begin{definition}
Given a sequent $\Gamma$ and countable fragment $H$, we write $\mathbb{DT}[\Gamma, H]$ to denote the $\omega$ branching tree of all deduction chains for $\Gamma$ and $H$. We call $\mathbb{DT}[\Gamma, H]$ the \emph{canonical tree} of $\Gamma,H$.
\end{definition}

\begin{remark}
Note that the tree $\mathbb{DT}[\Gamma, H]$ constitutes a cut-free pre-proof in our proof system. So if $\mathbb{DT}[\Gamma, H]$ is well-founded, then $\mathbb{DT}[\Gamma, H]$ constitutes a cut-free $\omega$-proof of $\Gamma$.
\end{remark}

The following standard lemma follows from the definition of deduction chains.

\begin{lemma}\label{standard_lemma}
Suppose $\mathbb{DT}[\Gamma, H]$ is ill-founded with path $\mathbb{P}$. Then:
\begin{enumerate}
\item $\mathbb{P}$ does not contain any literals that are true in $\mathbb{N}$.
\item $\mathbb{P}$ does not contain formulas $s\in K_i$ and $t\notin K_i$ for constant terms $s$ and $t$ such that $s^\mathbb{N}=t^\mathbb{N}$.
\item If $\mathbb{P}$ contains $E_0\vee E_1$, then $\mathbb{P}$ contains $E_0$ and $E_1$.
\item If $\mathbb{P}$ contains $E_0\wedge E_1$, then $\mathbb{P}$ contains $E_0$ or $E_1$.
\item If $\mathbb{P}$ contains $\exists xF(x)$, then $\mathbb{P}$ contains $F(\bar{n})$ for all $n$.
\item If $\mathbb{P}$ contains $\forall xF(x)$, then $\mathbb{P}$ contains $F(\bar{n})$ for some $n$.
\item If $\mathbb{P}$ contains $\exists XF(X)$, then $\mathbb{P}$ contains $F(U)$ for all set variables and constants $U$.
\item If $\mathbb{P}$ contains $\forall XF(X)$, then $\mathbb{P}$ contains $F(U)$ for some set variable/constant $U$.
\end{enumerate}
\end{lemma}
To see why clauses 1 and 2 of Lemma \ref{standard_lemma} are true, note that if $\mathbb{P}$ contained a true atomic sentence $\varphi$, then $\varphi$  would belong to an axiomatic sequent, but by definition deduction chains do not contain axiomatic sequents.

Now we are ready to prove our completeness theorems for $\omega$-models.

\begin{theorem}[$\mathsf{ACA}_0^+$]\label{omega_completeness_3}
For any closed sequent $\Gamma$ the following are equivalent:
\begin{enumerate}
    \item $\vdash_0 \Gamma$;
    \item There exists a family $S$ of sets such that for any $\omega$-model $\mathfrak{M}\supseteq S$ we have $\mathfrak{M}\models \Gamma$. 
\end{enumerate}
\end{theorem}

\begin{proof}
That (1) implies (2) follows from the soundness of the proof system with respect to $\omega$-models. 

For (2) implies (1) we prove the contrapositive. Assume that $\Gamma$ does not have a cut-free $\omega$-proof. Let $S$ be a family of sets and let $H$ be a countable fragment in which all sets in $S$ are named. Note that $\mathbb{DT}[\Gamma, H]$ is ill-founded; otherwise, it would constitute a cut-free $\omega$-proof of $\Gamma$. We will use an infinite path through $\mathbb{DT}[\Gamma, H]$ to define an $\omega$-model $\mathfrak{M}$ containing the sets named in $H$ (and so \emph{a fortiori} the sets in $S$) such that $\Gamma$ fails  in $\mathfrak{M}$.

Let $P$ be a path through $\mathbb{DT}[\Gamma, H]$, and let $\mathbb{P}$ be the set of all formulas that occur in $P$. For any set term (variable or constant) $K$, we now assign a subset $\mathsf{val}(K)$ of $\mathbb{N}$ to $K$ as follows: $$\mathsf{val}(K):=\{t^\mathbb{N}:\textrm{ $t$ is a constant $\mathbf{L}_2$ term and $(t\notin K)$ belongs to $\mathbb{P}\}$}.$$ It is easy to verify, given the axioms of our proof system, that for any $C_A\in H$, $\mathsf{val}(C_A)$ is the set $A$. 

Let $M$ be the weak $\omega$-model given by relativizing the second-order quantifiers to the disjoint union of the values $\mathsf{val}(K_n)$. Since we are reasoning in $\mathsf{ACA}_0^+$ we may enrich $M$ with a full satisfaction class, yielding an $\omega$-model $\mathfrak{M}$. An induction on the complexity of formulas (making use of Lemma \ref{standard_lemma}) shows that for any formula $\varphi$, $\varphi \in \mathbb{P}$ only if $\mathfrak{M}\nvDash \varphi$. Thus, the assumption that $\mathbb{DT}[\Gamma, H]$ is ill-founded implies that there is an $\omega$-model $\mathfrak{M}$ containing each set named by a constant in $H$ in which every sentence in $\Gamma$ is false.
\end{proof}

\subsection{Consequences of the completeness theorems}

A weak $\omega$-model $M$ is an at most countable set of subsets of $\mathbb{N}$ that is interpreted as the range for second-order variables. In $\mathsf{ACA}_0^+$ given a weak $\omega$-model we could always expand it by its unique full satisfaction class and thus obtain an $\omega$-model. However, we cannot do this over $\mathsf{ACA}_0$. Instead, in $\mathsf{ACA}_0$ we can form relativizations $(\varphi(\vec{X},\vec{x}))^M$ of $L_2$-formulas $\varphi$ to weak $\omega$-models. The formula $(\varphi(\vec{X},\vec{x}))^M$ is the result of replacement of second-order quantifiers $\forall Y$ with $\forall Y\dot\in M$.

Note that the proof of Theorem \ref{omega_completeness_3} goes through entirely in $\mathsf{ACA}_0$, except for the appeal to $\mathsf{ACA}_0^+$ to enrich $M$ with a full satisfaction class. Thus, the same proof yields the following version of the theorem:

\begin{theorem}\label{omega_completeness_2} Suppose $\Gamma(\vec{X})$ is a finite $L_2$-sequent. Then $\mathsf{ACA}_0$ proves that the following are equivalent for any $\vec{X}$:
\begin{enumerate}
    \item $\vdash_0 \Gamma(\dot{\vec{X}})$;
    \item there exists a family $S$ of sets such that for any $M\dot\supseteq S$ if $\vec{X}\dot\in M$, then $(\bigvee \Gamma(\vec{X}))^M$. 
\end{enumerate}
\end{theorem}

The following is a standard fact about $\mathsf{ACA}_0^+$.
\begin{proposition}[$\mathsf{ACA}_0$]\label{ACA_0^+_alt} The following are equivalent:
\begin{enumerate}
    \item $\mathsf{ACA}_0^+$;
    \item for any set $X$ there is $M$ such that $X\dot\in M$ and $(\mathsf{ACA}_0)^M$.
\end{enumerate}
\end{proposition}

\begin{theorem}[$\mathsf{ACA}_0$]\label{omegaRFN_eqiv1} For any $\mathbf{L}_2$-theory $T$ the following assertions are equivalent:
\begin{enumerate}
    \item \label{omegaRFN_eqiv1_1} $\mathbf{\Pi}^1_1\mbox{-}\omega\mathsf{RFN}(T)$;
    \item \label{omegaRFN_eqiv1_2} $\mathbf{\Pi}^1_2\mbox{-}\omega\mathsf{RFN}(T)$;
    \item \label{omegaRFN_eqiv1_3} $\forall X\exists \mathfrak{M}\;\big(X\in \mathfrak{M}\land \mathfrak{M}\models T\big)$;
    \item \label{omegaRFN_eqiv1_4} $\nvdash_\omega \lnot T$;
    \item \label{omegaRFN_eqiv1_5} $\nvdash_0 \lnot T$;
    \item \label{omegaRFN_eqiv1_6} $\forall \varphi\in \mathbf{\Pi}^1_2\;\big((\vdash_\omega \lnot T,\varphi)\to \mathsf{Tr}_{\mathbf{\Pi}^1_2}(\varphi)\big)$;
    \item \label{omegaRFN_eqiv1_7} $\forall \varphi\in \mathbf{\Pi}^1_2\;\big((\vdash_0 \lnot T,\varphi)\to \mathsf{Tr}_{\mathbf{\Pi}^1_2}(\varphi)\big)$.
\end{enumerate}
\end{theorem}
\begin{proof}
Clearly, we have implications (\ref{omegaRFN_eqiv1_2}.$\Rightarrow$\ref{omegaRFN_eqiv1_1}.), (\ref{omegaRFN_eqiv1_2}.$\Rightarrow$\ref{omegaRFN_eqiv1_3}.), (\ref{omegaRFN_eqiv1_2}.$\Rightarrow$\ref{omegaRFN_eqiv1_4}.), (\ref{omegaRFN_eqiv1_2}.$\Rightarrow$\ref{omegaRFN_eqiv1_6}.), (\ref{omegaRFN_eqiv1_2}.$\Rightarrow$\ref{omegaRFN_eqiv1_7}.) and (\ref{omegaRFN_eqiv1_1}.$\Rightarrow$\ref{omegaRFN_eqiv1_5}.), (\ref{omegaRFN_eqiv1_3}.$\Rightarrow$\ref{omegaRFN_eqiv1_5}.), (\ref{omegaRFN_eqiv1_4}.$\Rightarrow$\ref{omegaRFN_eqiv1_5}.), (\ref{omegaRFN_eqiv1_6}.$\Rightarrow$\ref{omegaRFN_eqiv1_5}.),
(\ref{omegaRFN_eqiv1_7}.$\Rightarrow$\ref{omegaRFN_eqiv1_5}.).
Henceforth, it is enough to prove that \ref{omegaRFN_eqiv1_5}. implies \ref{omegaRFN_eqiv1_2}.

Indeed, let us assume \ref{omegaRFN_eqiv1_5}. Since $T$ contains $\mathsf{ACA}_0$, we have $\nvdash_0\lnot \mathsf{ACA}_0$. Thus by Theorem \ref{omega_completeness_2}, for any set $X$ there is $M$ such that $X\dot\in M$ and $(\mathsf{ACA}_0)^M$. Thus by  Proposition \ref{ACA_0^+_alt} we have $\mathsf{ACA}_0^+$. By Theorem \ref{omega_completeness_3}, we see that there are arbitrarily large $\omega$-models of $T$. Using $\mathsf{ACA}_0^+$ we easily show that any false $\mathbf{\Pi}^1_2$ sentence $\varphi$ fails in all large enough $\omega$-models of $\mathsf{ACA}_0$. Combining the latter two facts we get $\mathbf{\Pi}^1_2\mbox{-}\omega\mathsf{RFN}(T)$.
\end{proof}

The same argument yields the following:
\begin{theorem}\label{omegaRFN_eqiv2} Let $n\ge 2$ be a natural number. Then $\mathsf{ACA}_0$ proves that for an $\mathbf{L}_2$-theory $T$ the following is equivalent:
\begin{enumerate}
    \item \label{omegaRFN_eqiv2_1} $\mathbf{\Pi}^1_n\mbox{-}\omega\mathsf{RFN}(T)$;
    \item \label{omegaRFN_eqiv2_2} $\forall \varphi\in \mathbf{\Pi}^1_2\;\big((\vdash_\omega \lnot T,\varphi)\to \mathsf{Tr}_{\mathbf{\Pi}^1_n}(\varphi)\big)$;
    \item \label{omegaRFN_eqiv2_3} $\forall \varphi\in \mathbf{\Pi}^1_2\;\big((\vdash_0 \lnot T,\varphi)\to \mathsf{Tr}_{\mathbf{\Pi}^1_n}(\varphi)\big)$.
\end{enumerate}
\end{theorem}

\section{Reduction for $\omega$-model reflection}\label{reduction}

In this section we prove the main result of this paper. First we prove a lemma, which can be viewed as an analogue of Feferman's completeness theorem for iterated $\Pi^1_1$ reflection. Then we prove Theorem \ref{omega_to_iter} (Theorem \ref{reduction_for_omega_models} in the introduction), which provides a reduction of $\omega$-model reflection to iterated syntactic reflection.

\subsection{An analogue of Feferman's theorem}

The $\omega$-rule provides one route to proving all arithmetical truths; indeed, the recursive $\omega$-rule suffices as shown by Shoenfield in \cite{shoenfield1969restricted}. Feferman provided another route in \cite{feferman1962transfinite}. Recall that for a theory $T$ in the language of first-order arithmetic, the uniform reflection schema $\mathsf{RFN}(T)$ for $T$ is the set of all sentences of the form:
$$ \forall \vec{x} \Big(\mathsf{Pr}_T\big(\varphi(\vec{x})\big) \rightarrow \varphi(\vec{x}) \Big)$$
where $\varphi(\vec{x})$ is a formula in the language of first-order arithmetic. Given an effective ordinal notation system $\prec$ we may then use the fixed point lemma to define the iterates of uniform reflection as follows:
\begin{flalign*}
\mathsf{RFN}^0(T) &: = T \\
\mathsf{RFN}^\alpha(T) &: = T + \bigcup_{\beta\prec\alpha}\mathsf{RFN}\big(\mathsf{RFN}^\beta(T)\big)\textrm{ for $\alpha\succ 0$.}
\end{flalign*}

\begin{theorem}[Feferman]
For any true arithmetical sentence $\varphi$, there is a representation $\alpha$ of a recursive ordinal such that $\mathsf{PA}+\mathsf{RFN}^\alpha(\mathsf{PA})\vdash \varphi$.
\end{theorem}

Feferman's proof makes crucial use of Shoenfield's completeness theorem for the recursive $\omega$-rule. In particular, Feferman shows that applications of the recursive $\omega$-rule can be simulated by iterating uniform reflection along a carefully selected ordinal notation. In \cite{schmerl1982iterated}, Schmerl cites this result (among others) as evidence that the uniform reflection principle is a formalized analogue of the $\omega$-rule.

In this subsection we will show that if a sequent of $\mathbf{\Pi}^1_n$ formulas can be proved from a $\mathbf{\Pi}^1_{n+1}$ axiomatized theory $T$ by applying the $\omega$-rule, then it can also be proved by iterating $\mathbf{\Pi}^1_n$ reflection. Thus, our main lemma is an analogue of Feferman's completeness theorem.

\begin{lemma}[$\mathsf{ACA}_0$]\label{feferman}
Let $n > 1$. Suppose that $T$ is a $\mathbf{\Pi}^1_{n+1}$ axiomatized theory, $\Gamma$ is a sequent of $\mathbf{\Pi}^1_n$ formulas, and $P$ is a cut-free $\omega$-proof of $\neg T, \Gamma$ with Kleene-Brouwer rank $\delta$. Then $\mathbf{\Pi}^1_n$-$\mathbf{R}^\delta(T)\vdash \bigvee\Gamma$.
\end{lemma}

\begin{proof}

Let $\Lambda$ be the statement of the lemma. We will prove $\Lambda$ by L\"{o}b's Theorem. That is, we will work in $\mathsf{ACA}_0$ and prove the statement $\mathsf{Pr}_{\mathsf{ACA}_0}(\Lambda)\rightarrow \Lambda$. It will then follow by L\"{o}b's Theorem that $\mathsf{ACA}_0$ proves $\Lambda$. 

So work in $\mathsf{ACA}_0$ and suppose that the statement of the lemma is provable in $\mathsf{ACA}_0$. Let $T$ and $\Gamma$ be as in the statement of the theorem. Let $\delta$ be the Kleene-Brouwer rank of the canonical tree $P$ for $\neg T, \Gamma$. We split into cases based on the final rule applied in $P$.

In each case $\neg T,\Gamma$ is being inferred from a sequence of sequents $\Delta_i$ which are the conclusions of canonical trees with Kleene-Brouwer ranks $\delta_i<\delta$. Our initial assumption that the statement of the lemma is provable in $\mathsf{ACA}_0$ yields that 
$$\mathsf{ACA}_0 \vdash \textrm{ ``for all $i$, $\mathbf{\Pi}^1_n$-$\mathbf{R}^{\delta_i}(T)$ proves $\bigvee \Delta_i$.''}$$ 
Which straightforwardly implies 
$$\mathsf{ACA}_0 \vdash \textrm{ ``for all $i$, if $\mathbf{\Pi}^1_n$-$\mathbf{R}^{\delta}(T)$ then $\mathsf{True}_{\mathbf{\Pi}^1_n} (\bigvee \Delta_i)$.''}$$ 
Which in turn implies
 $$\mathbf{\Pi}^1_n\text{-}\mathbf{R}^\delta(T)\vdash \forall i \mathsf{True}_{\mathbf{\Pi}^1_n} (\bigvee\Delta_i).$$ 
It suffices to check that this guarantees that $\mathbf{\Pi}^1_n$-$\mathbf{R}^\delta(T)\vdash \bigvee\Gamma.$

Since our canonical tree is cut-free, for each $i$, $\bigvee \Delta_i$ consists of $\mathbf{\Sigma}^1_{n+1}$ formulas (subformulas of negations of axioms of $T$) and $\mathbf{\Pi}^1_n$ formulas (subformulas of members of $\Gamma$). $\mathbf{\Pi}^1_n$-$\mathbf{R}^\delta(T)$ automatically rejects the negations of $T$'s axioms and so accepts the $\mathbf{\Pi}^1_n$ parts of these sequents (consisting only of subformulas of members of $\Gamma$). Then after checking, case-by-case, the soundness of each proof rule, $\mathbf{\Pi}^1_n$-$\mathbf{R}^\delta(T)$ infers $\bigvee\Gamma$ from $\forall i \mathsf{True}_{\mathbf{\Pi}^1_n} (\bigvee\Delta_i).$
\end{proof}




\subsection{The main theorem}

Now for the proof of the main theorem. We note that the general idea of the proof essentially is going back to Friedman's proof of equivalence of the scheme of bar induction and full scheme of $\omega$-model reflection \cite{friedman1975somesystems}.

\begin{theorem}\label{omega_to_iter}
Let $n > 0$. $\mathsf{ACA}_0$ proves that for any $\mathbf{\Pi}^1_{n+1}$-axiomatizable theory $T$, the following are equivalent:
\begin{enumerate}
\item\label{semantic} $\mathbf{\Pi}^1_n\textrm{-}\omega\mathsf{RFN}(T)$
\item\label{syntactic} $\forall \alpha \big(\mathsf{WO}(\alpha) \rightarrow \mathbf{\Pi}^1_n\textrm{-}\mathsf{RFN}(\mathbf{\Pi}^1_n\textrm{-}\mathbf{R}^\alpha(T))\big).$
\end{enumerate}
\end{theorem}
\begin{proof} The \ref{semantic}$\rightarrow$\ref{syntactic} direction is relatively straightforward. 

Assume, for contradiction, that \ref{semantic} is true but \ref{syntactic} is false. Since \ref{syntactic} is false, there is a well-ordering $\alpha$ such that $\mathbf{\Pi}^1_n\textrm{-}\mathsf{RFN}(\mathbf{\Pi}^1_n\textrm{-}\mathbf{R}^\alpha(T))$ is false. So for some false $\mathbf{\Pi}^1_n$ sentence $\varphi$,
\begin{equation}\label{contradiction1}
\mathbf{\Pi}^1_n\textrm{-}\mathbf{R}^\alpha(T) \vdash \varphi. 
\end{equation}
Note that $\neg\varphi$ is a true $\mathbf{\Sigma}^1_n$ statement. By \ref{semantic}, we infer that there is an $\omega$-model $\mathfrak{M}$ of $T$ such that:
\begin{equation}\label{contradiction2}
\mathfrak{M} \vDash \neg\varphi.
\end{equation}


On the other hand, by induction, we can show that $\mathfrak{M}$ satisfies $\mathbf{\Pi}^1_n\textrm{-}\mathbf{R}^\alpha(T)$. Assume that for every $\beta<\alpha$, $\mathfrak{M}\vDash \mathbf{\Pi}^1_n\textrm{-}\mathbf{R}^\beta(T)$. If $\mathfrak{M}\nvDash \mathbf{\Pi}^1_n\textrm{-}\mathbf{R}^\alpha(T)$ then 
$$\mathfrak{M}\vDash\exists \beta<\alpha\textrm{ ``}\mathbf{\Pi}^1_n\textrm{-}\mathbf{R}^\beta(T)\textrm{ proves a false }\mathbf{\Pi}^1_n\textrm{ statement }\psi.\textrm{''}$$ 
Since $\mathfrak{M}$ is an $\omega$-model, it is correct about what is provable. That is, this claim must be witnessed in $\mathfrak{M}$ by a standard proof. However, for any $\beta<\alpha$ and $\mathbf{\Pi}^1_n$ statement $\psi$, if $\mathbf{\Pi}^1_n\textrm{-}\mathbf{R}^\beta(T)$ proves $\psi$ then since $\mathfrak{M}$ is a model of $\mathbf{\Pi}^1_n\textrm{-}\mathbf{R}^\beta(T)$, $\mathfrak{M}$ is a model of $\psi$, and thus, that $\psi$ is a true $\mathbf{\Pi}^1_n$ statement.

Thus, we conclude that
\begin{equation}\label{contradiction3}
\mathfrak{M} \vDash \mathbf{\Pi}^1_n\textrm{-}\mathbf{R}^\alpha(T)
\end{equation}
But \ref{contradiction1}, \ref{contradiction2}, and \ref{contradiction3} are jointly inconsistent.

The \ref{syntactic}$\rightarrow$\ref{semantic} direction is less straightforward, but we have already laid the groundwork. We assume \ref{syntactic}. We want to prove \ref{semantic}, i.e., that every true $\mathbf{\Sigma}^1_n$ sentence is satisfied by an $\omega$-model of $T$. By Proposition \ref{ACA_0^+_alt}, it suffices to prove that every true $\mathbf{\Sigma}^1_n$ sentence is satisfied by a weak $\omega$-model of $T$. So let $\varphi$ be a true $\mathbf{\Sigma}^1_n$ sentence. We want to show that $\varphi$ is satisfied by an $\omega$-model of $T$ with a partial satisfaction class for $\mathbf{\Pi}^1_n$ sentences. We break into cases based on whether there is a cut-free $\omega$-proof of $\neg T, \neg \varphi$. 

\textbf{Case I:} There is no such proof, i.e., $\nvdash_0 \neg T, \neg \varphi$. By Theorem \ref{omega_completeness_2}, for every family $S$ of sets there is a weak $\omega$-model $\mathfrak{M}\supseteq S$ satisfying $T+\varphi$. This yields \ref{semantic}.



\textbf{Case II:} There is such a proof, i.e., $\vdash_0 \neg T,\neg\varphi$. Let $\delta$ be the Kleene-Brouwer rank of the canonical proof tree of $\neg T,\neg \varphi$. By \ref{syntactic} we can iterate reflection along $\delta$, yielding $\mathbf{\Pi}^1_n\text{-}\mathsf{RFN}(\mathbf{\Pi}^1_n$-$\mathbf{R}^\delta(T))$. On the other hand, by Lemma \ref{feferman}, $\mathbf{\Pi}^1_n$-$\mathbf{R}^\delta(T)\vdash \neg\varphi$. Combining these two observations, we conclude that $\neg \varphi$ is true, contradicting our choice of $\varphi$.
\end{proof}

As a special case we get the following:

\begin{theorem}[$\mathsf{ACA}_0$]\label{special_case}
For any $\mathbf{\Pi}^1_{2}$-axiomatizable theory $T$, the following are equivalent:
\begin{enumerate}
\item Every set is contained in an $\omega$-model of $T$.
\item $\mathbf{\Pi}^1_2\textrm{-}\omega\mathsf{RFN}(T)$.
\item $\forall \alpha \Big(\mathsf{WO}(\alpha) \rightarrow \mathbf{\Pi}^1_1\textrm{-}\mathsf{RFN}\big(\mathbf{\Pi}^1_1\textrm{-}\mathbf{R}^\alpha(T)\big)\Big).$
\item $\forall \alpha \Big(\mathsf{WO}(\alpha) \rightarrow \mathbf{\Pi}^1_2\textrm{-}\mathsf{RFN}\big(\mathbf{\Pi}^1_2\textrm{-}\mathbf{R}^\alpha(T)\big)\Big).$
\end{enumerate}
\end{theorem}


\section{Proof-theoretic dilators}\label{pt_dilators}

In this section we introduce the concept of the \emph{proof-theoretic dilator of a theory}. Proof-theoretic dilators play a role in $\Pi^1_2$ proof theory that is analogous to the role proof-theoretic ordinals play in $\Pi^1_1$ proof theory. In this section, we also use the main theorem to establish a systematic connection between iterated $\omega$-model reflection and the dilators of theories. In particular, we characterize how the proof-theoretic dilators of theories grow as a function of the amount of $\omega$-model reflection they prove.

\subsection{From rules to axioms}

\begin{lemma}[$\mathsf{ACA}_0$]\label{WO_NF} For any $\mathbf{\Pi}^1_1$-formula $\varphi(\vec{x},\vec{X})$ there is an elementary comprehension term $\alpha(\vec{x},\vec{X})$ such that $\mathsf{ACA}_0\vdash \varphi(\vec{x},\vec{X})\mathrel{\leftrightarrow}\mathsf{WO}(\alpha(\vec{x},\vec{X}))$. Furthermore we could choose  $\alpha$ so that any constant that appears in $\alpha$ also appears in $\varphi$.\end{lemma}
\begin{proof}
Since this is very close to a standard fact we don't give a detailed proof. This could be achieved by taking as $\alpha(\vec{x},\vec{X})$ the Kleene-Brouwer order on the Kleene normal form of $\varphi(\vec{x},\vec{X})$ (see \cite[Lemma~V.1.4]{simpson2009subsystems} for a presentation of Kleene's normal form theorem formalized in $\mathsf{ACA}_0$). 
\end{proof}

\begin{lemma}[$\mathsf{ACA}_0$]\label{disjunction_of_orders} There is an elementary comprehension term $\mathsf{disj}(X,Y)$ such that $\mathsf{ACA}_0$ proves that for any linear orders $\alpha,\beta$:
\begin{enumerate}
    \item $\mathsf{disj}(\alpha,\beta)$ is a linear order;
    \item $\mathsf{WO}(\mathsf{disj}(\alpha,\beta))\mathrel{\leftrightarrow} \mathsf{WO}(\alpha)\lor \mathsf{WO}(\beta)$;
    \item if there is an infinite descending chain $b_0\succ_\beta b_1\succ_\beta \ldots$ then there is an embedding $f\colon \alpha\to \mathsf{disj}(\alpha,\beta)$.
\end{enumerate}
\end{lemma}
\begin{proof}
   Let us fix orders $\alpha$, $\beta$ and describe the order $\mathsf{disj}(\alpha,\beta)=\gamma$. Let $\delta$ be the partial order that is the product of $\alpha$ and $\beta$ as partial orders, i.e., $\delta$ consists of pairs $\langle a,b\rangle$ where $a\in\alpha$, $b\in\beta$ and we have $\langle a_1,b_1\rangle\prec_\delta \langle a_2,b_2\rangle$ iff $a_1\prec_\alpha a_2$ and $b_1\prec_\delta b_2$. The domain  of the order $\gamma$ consists of all sequences $c=(c_0,c_1,\ldots,c_{n-1})$ such that $c_0\succ_\delta c_1\succ_\delta\ldots \succ_\delta c_{n-1}$. We put $c= (c_0,c_1,\ldots,c_{n-1})\prec_\gamma (c_0',c_1',\ldots,c_{m-1}')=c'$  if either $c'$ is a proper initial segment of $c$ or if $c_i<c_i'$ (this is comparison of $c_i$ and $c_i'$ as natural numbers), where $i$ is the least index such that $c_i\ne c_{i}'$. 
   It is fairly easy to see that this construction of $\gamma$ in fact could be given by an elementary comprehension term. Note $\mathsf{ACA}_0$ proves that $\gamma$ is a linear order. Moreover, provably in $\mathsf{ACA}_0$, we have $\mathsf{WF}(\delta)\mathrel{\leftrightarrow}\big(\mathsf{WO}(\alpha)\lor \mathsf{WO}(\beta)\big)$ and $\mathsf{WF}(\delta)\mathrel{\leftrightarrow} \mathsf{WO}(\gamma)$.
   
   Now we just need to show in $\mathsf{ACA}_0$ that given an infinite descending chain $b_0\succ_\beta b_1\succ_\beta \ldots$, there is an embedding $f\colon \alpha\to \mathsf{disj}(\alpha,\beta)$. We enumerate the elements of $\alpha$ as $a_0,a_1,\ldots$ (if $\alpha$ is finite then the list would be finite).  We put $f(a_0)=\langle a_0,b_0\rangle$. For each next $a_{i+1}$ we consider two cases:
   \begin{enumerate}
       \item $a_{i+1}\succ_{\alpha}a_j$, for each $j\le i$;
       \item $a_{i+1}\prec_{\alpha}a_j$, for some $j\le i$.
   \end{enumerate}
   In the case 1. we put $f(a_{i+1})=\langle a_0,b_k\rangle$, where we choose $k$ such that $\langle a_0,b_k\rangle$ would be large enough as a number so that $f(a_{i+1})\succ_\gamma f(a_j)$, for $j\le i$. In the case 2. we consider $a_v=\min_\alpha\{a_j\mid j\le i \text{ and }a_{i+1}\prec_{\alpha}a_j\}$. The value $f(a_v)$ is of the form $(c_0,\ldots,c_{m-1})$, where $c_{m-1}$ is of the form $\langle a_v, b_u\rangle$. We put $f(a_{i+1})$ to be of the form $(c_0,\ldots,c_{m-1},\langle a_{i+1},b_k\rangle)$, where we choose $k$ such that $k>u$ and $\langle a_{i+1},b_k\rangle$ as a number is large enough so that $f(a_{i+1})\succ_\gamma f(a_j)$ for all $a_j\prec a_{i+1}$, $j\le i$.
   It is easy to see that this construction gives us an embedding of $\alpha$ into $\gamma$.
\end{proof}

\begin{lemma}[$\mathsf{ACA}_0$]
For any elementary comprehension term $t(X_1,\ldots,X_n,\vec{y})$ there is an elementary comprehension term $\mathsf{appr}_t(X_1,\ldots,X_n,\vec{y})$ with the same parameters such that $\mathsf{ACA}_0$ proves the following. If for any $X_1,\ldots,X_n$, and $\vec{y}$ we have $\mathsf{LO}(t(X_1,\ldots,X_n,\vec{y}))$, then \begin{enumerate} \item $\forall x_1,\ldots,x_n\in 2^{<\omega}\forall \vec{y}\Big(\mathsf{LO}\big(\mathsf{appr}_t(x_1,\ldots,x_n,\vec{y})\big)\Big)$, \item for any $\vec{y}$, and $x_1,\ldots,x_n,x_1',\ldots,x_n'\in 2^{<\omega}$ if $x_1\subseteq x_1',\ldots,x_n\subseteq x_n'$, then $\mathsf{appr}_t(x_1,\ldots,x_n,\vec{y})\subseteq \mathsf{appr}_t(x_1',\ldots,x_n',\vec{y})$, \item for any $\vec{y}$ and $X_1,\ldots,X_n$ we have $t(X_1,\ldots,X_n,\vec{y})=\bigcup\limits_{m<\omega} \mathsf{appr}_t(X_1\upharpoonright m,\ldots,X_n\upharpoonright m,\vec{y})$.\end{enumerate}
\end{lemma}
\begin{proof}
Observe that there is an iteration of exponentiation $2_s^x$ such that for any $k$ and $i,j\le k$ when we calculate whether $\langle i,j\rangle\in t(X_1,\ldots,X_n,\vec{y})$ we make the requests about membership in $X_1,\ldots,X_n$ only for numbers $l<2_s^k$. Hence in fact we have an elementary comprehension term $v(x_1,\ldots,x_n,\vec{y},z)$ such that for any $k$ and $X_1,\ldots,X_n$ we have $$v(X_1\upharpoonright 2_s^k,\ldots,X_n\upharpoonright 2_s^k,\vec{y},k)=\{\langle i,j\rangle\mid \langle i,j\rangle \in t(X_1,\ldots,X_n,\vec{y}) \text{ and } i,j\le k\}.$$ 

Let $\log_2^s(x)$ be the function mapping natural number $x$ to the greatest natural number $y$ such that $2^x_s\ge y$. Let $f(x)=2_s^{\log_2^s(x)}$. Let $|x|$ be the function mapping $x\in 2^{<\omega}$ to its length.  Let $h(x_1,\ldots,x_n)=f\big(\min(|x_1|,\ldots,|x_n|)\big)$. Clearly these functions are elementary recursive.

We put $\mathsf{appr}_t(x_1,\ldots,x_n,\vec{y})$ to be $$v(x_1\upharpoonright h(x_1,\ldots,x_n),\ldots,x_n\upharpoonright h(x_1,\ldots,x_n),\vec{y},\log_2^s(\min(|x_1|,\ldots,|x_n|))).$$
It is easy to see that the term $\mathsf{appr}_t$ behaves as desired.
\end{proof}

\begin{lemma}[$\mathsf{ACA}_0$] \label{mrj_lemma}
For any elementary comprehension term $t(X_1,\ldots,X_n,\vec{y})$ there is an elementary comprehension term $\mathsf{mrj}_t$ and an arithmetical comprehension term $\mathsf{emb}_t(X_1,\ldots,X_n,\vec{y})$ with the same parameters such that $\mathsf{ACA}_0$ proves the following. If for any $X_1,\ldots,X_n$, and $\vec{y}$ we have $\mathsf{LO}(t(X_1,\ldots,X_n,\vec{y}))$, then \begin{enumerate} \item $\mathsf{LO}(\mathsf{mrj}_t)$; \item if $\mathsf{WO}(\mathsf{mrj}_t)$, then $\mathsf{emb}_t(X_1,\ldots,X_n,\vec{y})$ is an embedding of $t(X_1,\ldots,X_n,\vec{y})$ into $\mathsf{mrj}_t$; \item\label{Mrj_lemma_c3} $\mathsf{WO}(\mathsf{mrj}_t)$ iff $\forall X_1,\ldots,X_n,\vec{y} \;\mathsf{WO}\big(t(X_1,\ldots,X_n,\vec{y})\big)$. \end{enumerate}
\end{lemma}
\begin{proof}
  Assuming  $t(X_1,\ldots,X_n,\vec{y})$ is a term defining orders for all $X_1,\ldots,X_n$, and $\vec{y}$ the term $\mathsf{mrj}_t$ should represent the following order. The domain for $\mathsf{mrj}_t$ should consist of the sequences $\langle \vec{p},\langle s_{0,1},\ldots,s_{0,n},a_0\rangle,\ldots, \langle s_{m-1,1},\ldots,s_{m-1,n},a_{m-1}\rangle\rangle$, where
  \begin{enumerate}
      \item the vector $\vec{p}$ is a vector of naturals of the same dimension as $\vec{y}$,
      \item all $s_{i,j}\in 2^{<\omega}$,
      \item $s_{i,j}\subsetneq s_{i+1,j}$, for all $0\le i<m-1$ and $1\le j\le n$,
      \item $a_{i}$ is from the domain of $\mathsf{appr}_t(s_{i,1},\ldots,s_{i,n},\vec{p})$,
      \item $a_{i+1}<a_i$ according to the order  $\mathsf{appr}_t(s_{i+1,1},\ldots,s_{i+1,n},\vec{p})$.
  \end{enumerate}
  We treat $\mathsf{mrj}_t$ as a subtree of $\omega^{<\omega}$ by identifying the individual components of sequences from the domain of $\mathsf{mrj}_t$ with the natural numbers coding them. The linear order on $\mathsf{mrj}_t$ is simply the Kleene-Brouwer order on this tree.
  
  It is easy to see that any infinite path in $\mathsf{mrj}_t$ as a tree induces a descending chain through some $t(X_1,\ldots,X_n,\vec{y})$ and a descending chain in any $t(X_1,\ldots,X_n,\vec{y})$ yields an inifinite path through $\mathsf{mrj}_t$. Hence, $\mathsf{mrj}_t$ is well-ordered iff each  $t(X_1,\ldots,X_n,\vec{y})$  is well-ordered, i.e., condition \ref{Mrj_lemma_c3}. holds. 
  
  A simulation of a linear order $\alpha$ in a linear order $\beta$ is a binary relation $S\subseteq D_\alpha\times D_\beta$ such that 
  \begin{enumerate} \item for any $a\in D_\alpha$ there is some $b\in D_{\beta}$ for which $a\mathrel{S} b$ and \item whenever $ a \mathrel{S} b$ and $a >_{\alpha} a'$, there is $b'\in D_\beta$ such that $a' \mathrel{S} b'$ and $b >_{\beta} b'$.\end{enumerate} Let $\mu S\colon D_R\to D_{R'}$ be the partial function from $D_\alpha$ to $D_\beta$ that maps $a$ to $\inf \{b\in D_{\beta}\mid a\mathrel{S} b\}$. It is easy to see that if $\beta$ is a well-order, then $\mu S$ is an embedding of $\alpha$ into $\beta$.
  
  Let us define a simulation $S(X_1,\ldots,X_n,\vec{y})$ of $t(X_1,\ldots,X_n,\vec{y})$ in $\mathsf{mrj}_t$. We put $a$ to be $S(X_1,\ldots,X_n,\vec{y})$-simulated by $$\langle \vec{y},\langle s_{0,1},\ldots,s_{0,n},a_0\rangle,\ldots, \langle s_{m-1,1},\ldots,s_{m-1,n},a_{m-1}\rangle\rangle$$ iff  $a=a_{m-1}$. It is trivial to see that this indeed forms a family of simulation relations. We put $\mathsf{emb}_t(X_1,\ldots,X_n,\vec{y})=\mu S(X_1,\ldots,X_n,\vec{y})$.
\end{proof}

\begin{lemma}\label{wo_cons}
Suppose $T$ is a $\mathbf{\Pi}^1_2$-axiomatizable $\mathbf{L}_2$-theory and $\varphi(X)$ is a $\mathbf{\Pi}^1_2$-formula such that $T\vdash \forall X,Y\big(\mathsf{WO}(X)\land \exists f(f\colon X\to Y)\land \varphi(Y)\to\varphi(X)\big)$. And suppose $H$ is a countable fragment of $\mathbf{L}_2$ containing $T$ and $\varphi$. Then the following theories have the same $\mathbf{\Pi}^1_1\cap H$ theorems:
\begin{enumerate}
    \item \label{wo_cons_1}$T+\forall X\big(\mathsf{WO}(X)\to \varphi(X)\big)$;
    \item \label{wo_cons_2}the closure of $T$ under the rules $\displaystyle\frac{\mathsf{WO}(\alpha)}{\varphi(\alpha)}$, where $\alpha$ ranges over closed elementary comprehension terms with constants from $H$;
    \item  \label{wo_cons_3}the closure of $T$ under the rules $\displaystyle\frac{\mathsf{WO}(\alpha)}{\varphi(\alpha)}$, where $\alpha$ ranges over closed arithmetical comprehension terms with constants from $H$.
\end{enumerate}
\end{lemma}
\begin{proof}
Clearly the theory from (\ref{wo_cons_1}) contains the theory from (\ref{wo_cons_3}), which in turn contains the theory from (\ref{wo_cons_2}). Thus it is enough to show that (\ref{wo_cons_1}) is $\mathbf{\Pi}^1_1\cap H$-conservative over  (\ref{wo_cons_2}). For this we will use Lemma \ref{rule_cons}. That is, it is enough to show that the theory (\ref{wo_cons_2}) is closed under all of the rules
$$\frac{\forall X,\vec{y},\vec{Y}(\theta(X,\vec{y},\vec{Y})\lor \mathsf{WO}(X)}{\forall X,\vec{y},\vec{Y}(\theta(X,\vec{y},\vec{Y})\lor\varphi(X))},$$
where $\theta\in \mathbf{\Pi}^1_1\cap H$ and $\alpha$ is a closed elementary comprehension term with constants from $H$. Let us denote the theory from (\ref{wo_cons_2}) as $\mathsf{U}$.  Further we fix one of the rules of the form above, assume that $U\vdash \forall X,\vec{y},\vec{Y}(\theta(X,\vec{y},\vec{Y})\lor \mathsf{WO}(X))$ and claim that $U\vdash  \forall X,\vec{y},\vec{Y}(\theta(X,\vec{y},\vec{Y})\lor\varphi(X))$.

By Lemma \ref{WO_NF} there is an elementary comprehension term $\beta(\vec{y},\vec{Y})$ such that $\mathsf{ACA}_0\vdash \forall \vec{y},\vec{Y}(\mathsf{WO}(\beta(X,\vec{y},\vec{Y}))\mathrel{\leftrightarrow} \theta(X,\vec{y},\vec{Y}))$. Let $\gamma$ be $\mathsf{mrj}_{\mathsf{disj}(X,\beta(X,\vec{y},\vec{Y}))}$. Observe that 
$$\mathsf{ACA}_0\vdash \mathsf{WO}(\gamma)\mathrel{\leftrightarrow} \forall X,\vec{y},\vec{Y}\big(\theta(X,\vec{y},\vec{Y})\lor \mathsf{WO}(X)\big)$$ and hence $U\vdash \mathsf{WO}(\gamma)$. Thus using the definition of $U$ we get $U\vdash \varphi(\gamma)$. 

To finish the proof we reason in $U$ and claim that $\forall X, \vec{y},\vec{Y}(\theta(\vec{y},\vec{Y})\lor\varphi(X))$. Indeed, we fix $A,\vec{b},\vec{B}$, assume that $\lnot \theta(A,\vec{b},\vec{B})$ and need to prove $\varphi(A)$. We have $\mathsf{WO}(A)$. Since we have $\mathsf{WO}(\gamma)$, by Lemma \ref{mrj_lemma}, $\mathsf{emb}_{\mathsf{disj}(A,\beta(X,\vec{y},\vec{Y}))}(A,\vec{b},\vec{B})$ is an embedding of $\mathsf{disj}(A,\beta(A,\vec{b},\vec{B}))$ into $\gamma$. Hence $\mathsf{WO}(\mathsf{disj}(A,\beta(A,\vec{b},\vec{B})))$ and  $\varphi(\mathsf{disj}(A,\beta(A,\vec{b},\vec{B})))$. Because $\lnot \theta(A,\vec{b},\vec{B})$, we have $\lnot \mathsf{WO}(\beta(A,\vec{b},\vec{B}))$. Thus there is an embedding of $A$ into $\mathsf{disj}(A,\beta(A,\vec{b},\vec{B}))$. Therefore we have $\varphi(A)$.
\end{proof}

Combining Lemma \ref{wo_cons}, Theorem \ref{special_case}, and Lemma \ref{hom_lemma} we obtain: 
\begin{lemma}[$\mathsf{ACA}_0$]\label{ON_iter_to_rule}
Let $T$ be some $\mathbf{\Pi}^1_2$-axiomatizable $\mathbf{L}_2$-theory.  Then for any countable fragment $H$ of $\mathbf{L}_2$ containing all axioms of $T$ the following three theories have the same $\mathbf{\Pi}^1_1\cap H$ theorems:
\begin{enumerate}
    \item \label{ON_iter_to_rule_t1}$T+\mathbf{\Pi}^1_2\textrm{-}\omega\mathsf{RFN}(T)$;
    \item \label{ON_iter_to_rule_t2}the closure of $T$ under the rule $$\frac{\mathsf{WO}(\alpha)}{\mathbf{\Pi}^1_1\mbox{-}\mathsf{RFN}(\mathbf{\Pi}^1_1\mbox{-}\mathbf{R}^{\alpha}(T))}\text{, where $\alpha$ is an arithmetical term with constants from $H$}.$$
    \item \label{ON_iter_to_rule_t3}the closure of $T$ under the rule $$\frac{\mathsf{WO}(\alpha)}{\mathbf{\Pi}^1_2\mbox{-}\mathsf{RFN}(\mathbf{\Pi}^1_2\mbox{-}\mathbf{R}^{\alpha}(T))}\text{, where $\alpha$ is an arithmetical term with constants from $H$}.$$
\end{enumerate}
\end{lemma}

Using Lemma \ref{wo_cons} and Lemma \ref{hom_lemma} we get:
\begin{lemma}\label{ON_omega_iter_to_rule}
Let $T$ be some $\mathbf{\Pi}^1_2$-axiomatizable $\mathbf{L}_2$-theory.  Then for any countable fragment $H$ of $\mathbf{L}_2$ containing all axioms of $T$ the following two theories have the same $\mathbf{\Pi}^1_1\cap H$ theorems:
\begin{enumerate}
    \item \label{ON_omega_iter_to_rule_t1}$T+(\forall \alpha)(\mathsf{WO}(\alpha)\to \mathbf{\Pi}^1_2\textrm{-}\omega\mathsf{RFN}(\mathbf{\Pi}^1_2\textrm{-}\omega\mathbf{R}^\alpha(T)))$;
\item \label{ON_omega_iter_to_rule_t3}the closure of $T$ under the rule $$\frac{\mathsf{WO}(\alpha)}{\mathbf{\Pi}^1_2\mbox{-}\omega\mathsf{RFN}(\mathbf{\Pi}^1_2\mbox{-}\omega\mathbf{R}^{\alpha}(T))}\text{, where $\alpha$ is an arithmetical term with constants from $H$}.$$
\end{enumerate}
\end{lemma}

\subsection{Proof-theoretic dilators}

Working in a strong meta-theory for a (countable) linear order $\alpha$ we write $|\alpha|$ where $|\alpha| \in\omega_1\cup\{\infty\}$ to denote its well-founded rank ($\infty$ is the rank of ill-founded orders, $\infty$ is greater than any ordinal). For an $\mathbf{L}_2$-theory $T$ we write $|T|_{\mathbf{\Pi}^1_1}$ to denote its proof-theoretic ordinal which we define as the supremum of ranks of the $T$-provably well-ordered $\mathbf{\Pi}^1_0$ linear orders. 

Unlike many other works on proof theoretic analysis in this paper in fact we will need a formalization of the notion of proof-theoretic ordinal in $\mathsf{ACA}_0$ rather than in an informal set-theoretic setting as we have done above. There are a few limitations that we need to address. First $\mathsf{ACA}_0$ doesn't have a good theory of ordinals. In particular it is known that $\mathsf{ATR}_0$ is equivalent to the second-order sentence ``for any well-orders $\alpha$ and $\beta$ either there is an isomorphism between $\alpha$ and an initial segment of $\beta$ or an isomorphism between $\beta$ and an initial segment of $\alpha$'' \cite[Theorem~V.6.8]{simpson2009subsystems}. Second, in $\mathsf{ACA}_0$ we need to be more careful when working with $\mathbf{\Pi}^1_0$-definable linear orders. We make the following definitions in $\mathsf{ACA}_0$.

The comparisons of ranks of linear orders:
\begin{itemize}
    \item $|\alpha|\le|\beta|$ if either there is a homomorphism $f\colon \alpha\to\beta$ or $\beta$ is ill-founded;
    \item $|\alpha|<|\beta|$ if there is a homomorphism $f\colon \alpha\to\mathsf{cone}(\beta,n)$, for some $n\in\beta$;
    \item $|\alpha|=|\beta|$ if $\alpha\le \beta$ and $\beta\le \alpha$.
\end{itemize}

\begin{remark} The authors do not know the reverse mathematical status of the sentence ``for any two well-orders $\alpha,\beta$ either $|\alpha|\le |\beta|$ or $|\beta|\le |\alpha|$'' other than the fact that it is provable in $\mathsf{ATR}_0$.\end{remark}

A $\mathbf{\Pi}^1_0$ linear order ${\boldsymbol\alpha}$ is a triple $\langle  D_{{\boldsymbol\alpha}}, \prec_{{\boldsymbol\alpha}},\models_{{\boldsymbol\alpha}}\rangle$ such that
\begin{itemize}
    \item $x\prec_{{\boldsymbol\alpha}} y$ and $D_{{\boldsymbol\alpha}}(x)$ are $\mathbf{\Pi}^1_0$-formulas without other free variables;
    \item $\models_{{\boldsymbol\alpha}}$ is a compositional partial satisfaction relation that is correct on atomic formulas and covers all subformulas of $\prec_{{\boldsymbol\alpha}}$ and $D_{{\boldsymbol\alpha}}$;
    \item the following binary relation $\mathbf{\alpha}^\star$ is a linear order: the domain of $\mathbf{\alpha}^\star$ is $D_{\mathbf{\alpha}^\star}=\{n \in \mathbb{N} \mid \; \models_{{\boldsymbol\alpha}} D_{{\boldsymbol\alpha}}(n)\}$ and $n\prec_{\mathbf{\alpha}^\star} m \defiff  \models_{{\boldsymbol\alpha}} n\prec_{{\boldsymbol\alpha}}m$ 
    is a linear order.
\end{itemize}
Note that $\models_{{\boldsymbol\alpha}}$ is included in the definition only due to the weakness of our base theory $\mathsf{ACA}_0$.  The order ${\boldsymbol\alpha}^\star$ is uniquely determined just by $D_{{\boldsymbol\alpha}}$, $\prec_{{\boldsymbol\alpha}}$; however we couldn't prove that for any $D_{{\boldsymbol\alpha}}$, $\prec_{{\boldsymbol\alpha}}$ there is a large enough partial satisfaction relation.

For a $\mathbf{\Pi}^1_0$ linear order ${\boldsymbol\alpha}$ the formula $\mathsf{WO}_{{\boldsymbol\alpha}}$ says that the binary relation given by the formulas $D_{{\boldsymbol\alpha}}$ and  $\prec_{{\boldsymbol\alpha}}$ is a well-ordering.
We write 
\begin{itemize}
    \item $|T|_{\mathbf{\Pi}^1_1}\le |\alpha|$ if for any $\mathbf{\Pi}^1_0$ linear order ${\boldsymbol\beta}$ we have
    $$T\vdash \mathsf{WO}_{{\boldsymbol\beta}}\Rightarrow |{\boldsymbol\beta}^\star|\le |\alpha|;$$
    \item $|T|_{\mathbf{\Pi}^1_1}\ge |\alpha|$ if for any $n\in \alpha$ there is a $\mathbf{\Pi}^1_0$ linear order ${\boldsymbol\beta}$ such that $T\vdash\mathsf{WO}_{{\boldsymbol\beta}}$ and $|{\boldsymbol\beta}^\star|\ge |\mathsf{cone}(\alpha,n)|$;
    \item $|T|_{\mathbf{\Pi}^1_1}= |\alpha|$ if $|T|_{\mathbf{\Pi}^1_1}\ge |\alpha|$ and $|T|_{\mathbf{\Pi}^1_1}\le |\alpha|$.    
\end{itemize}

\begin{proposition}[$\mathsf{ACA}_0$]
Let $T$ be a theory and $\alpha$ be a linear order. 
\begin{enumerate}
    \item If $|T|_{\mathbf{\Pi}^1_1}\le |\alpha|$ and $\mathsf{WO}(\alpha)$, then $\mathbf{\Pi}^1_1\textrm{-}\mathsf{RFN}(T)$.
    \item  If $|T|_{\mathbf{\Pi}^1_1}\ge |\alpha|$ and $\mathbf{\Pi}^1_1\textrm{-}\mathsf{RFN}(T)$, then $\mathsf{WO}(\alpha)$.
\end{enumerate}     
\end{proposition}
\begin{proof}
  First let us prove 1. We reason in $\mathsf{ACA}_0$ and assume $|T|_{\mathbf{\Pi}^1_1}\le |\alpha|$ and $\mathsf{WO}(\alpha)$. We claim that  $\mathbf{\Pi}^1_1\textrm{-}\mathsf{RFN}(T)$. Suppose that $\varphi$ is a $T$-provable $\mathbf{\Pi}^1_1$-sentence. We need to show that $\varphi$ is true. By relativized Kleene's normal form theory we could find a $\mathbf{\Pi}^1_0$ linear order ${\boldsymbol\beta}$ (in fact $\mathbf{\Delta}^0_1$ linear order) such that $\mathsf{Tr}_{\mathbf{\Pi}^1_1}(\varphi)\leftrightarrow \mathsf{WO}_{{\boldsymbol\beta}}$ and $\mathsf{ACA}_0\vdash \mathsf{Tr}_{\mathbf{\Pi}^1_1}(\varphi)\leftrightarrow \mathsf{WO}_{{\boldsymbol\beta}}$. Thus $T\vdash \mathsf{WO}_{{\boldsymbol\beta}}$. Hence $|{\boldsymbol\beta}^\star|\le |\alpha|$. Therefore $\mathsf{WO}_{{\boldsymbol\beta}}$ and thus $\mathsf{Tr}_{\mathbf{\Pi}^1_1}(\varphi)$.
  
  Now let us prove 2. We reason in in $\mathsf{ACA}_0$ and assume $|T|_{\mathbf{\Pi}^1_1}\ge |\alpha|$ and $\mathbf{\Pi}^1_1\textrm{-}\mathsf{RFN}(T)$. We claim that $\mathsf{WO}(\alpha)$. For this it is enough to show that any cone in $\alpha$ is well-ordered. Consider a cone $\mathsf{cone}(\alpha,n)$. For some $T$-provably well-ordered $\mathbf{\Pi}^1_0$ linear order  ${\boldsymbol\beta}$ we have $|{\boldsymbol\beta}^{\star}|\ge  |\mathsf{cone}(\alpha,n)|$. By $\mathbf{\Pi}^1_1\textrm{-}\mathsf{RFN}(T)$ we have $\mathsf{Tr}_{\mathbf{\Pi}^1_1}(\mathsf{WO}_{{\boldsymbol\beta}})$.  Thus we have $\mathsf{WO}_{{\boldsymbol\beta}^{\star}}$ and hence $\mathsf{WO}(\mathsf{cone}(\alpha,n))$.
\end{proof}

For theories $T$ and $U$ we write
\begin{enumerate}
    \item $|T|_{\mathbf{\Pi}^1_1}\le|U|_{\mathbf{\Pi}^1_1}$ if for any $\mathbf{\Pi}^1_0$ linear order ${\boldsymbol\alpha}$ if  $T\vdash \mathsf{WO}_{{\boldsymbol\alpha}}$, then $|{\boldsymbol\alpha}^\star|\le |U|_{\mathbf{\Pi}^1_1}$;
    \item $|T|_{\mathbf{\Pi}^1_1}=|U|_{\mathbf{\Pi}^1_1}$ if $|T|_{\mathbf{\Pi}^1_1}\le|U|_{\mathbf{\Pi}^1_1}$ and $|T|_{\mathbf{\Pi}^1_1}\ge|U|_{\mathbf{\Pi}^1_1}$.
\end{enumerate}

It is easy to see that according to oir definitions, provably in $\mathsf{ACA}_0$, the binary relation $\le$ on theories and linear orders is a (class-sized) transitive binary relation. 

\begin{remark} Although we don't prove this in the present paper, it is in fact easy to show that provably in $\mathsf{ACA}_0$ for any theory $T$ there is $\alpha$ such that $|T|_{\mathbf{\Pi}^1_1}=|\alpha|$. Namely one could take as $\alpha$ the ordered sum $\sum\limits_{n<\omega} {\boldsymbol\alpha}_n^{\star}$, where ${\boldsymbol\alpha}_0,{\boldsymbol\alpha}_1,\ldots$ is an enumeration of all $\mathbf{\Delta}^0_1$ linear orders such that $T\vdash \mathsf{WO}_{{\boldsymbol\alpha}_n}$ and the formulas $\prec_{{\boldsymbol\alpha}},D_{{\boldsymbol\alpha}}$ use only the set constants used in the axioms of $T$.
\end{remark}

 Note that by a classical result of Kreisel (see \cite[Theorem 6.7.4,6.7.5]{pohlers2008first}) for extensions of $\mathsf{ACA}_0$ the $\Pi^1_1$ proof theoretic ordinals are stable with respect to extensions by true $\Sigma^1_1$-sentences. We have the following variant of Kreisel's result:
\begin{proposition}[$\mathsf{ACA}_0$]
If $T\sqsubseteq^{\mathbf{\Sigma}^1_1}U$, then $|T|_{\mathbf{\Pi}^1_1}\le |U|_{\mathbf{\Pi}^1_1}$. And hence if $T\equiv^{\mathbf{\Sigma}^1_1}U$, then $|T|_{\mathbf{\Pi}^1_1}= |U|_{\mathbf{\Pi}^1_1}$.
\end{proposition}
\begin{proof}  We consider theories $T,U$ such that $T\sqsubseteq^{\mathbf{\Sigma}^1_1}U$ and claim that $|T|_{\mathbf{\Pi}^1_1}\le |U|_{\mathbf{\Pi}^1_1}$. For some true $\mathbf{\Sigma}^1_1$ sentence $\varphi$ we have $T\sqsubseteq U+\varphi$. Using relativized Kleene's normal form theorem we find an ill-founded $\mathbf{\Pi}^1_0$ linear order ${\boldsymbol\alpha}$ such that $\mathsf{ACA}_0\vdash \lnot \mathsf{WO}_{{\boldsymbol\alpha}}\mathrel{\leftrightarrow} \varphi$. We need to show that for any given $T$-provably well-ordered $\mathbf{\Pi}^1_0$ linear order ${\boldsymbol\beta}$ we have  $|{\boldsymbol\beta}^\star|\le |{\boldsymbol\gamma}^{\star}|$ for some $U$-provably well-founded $\mathbf{\Pi}^1_0$ linear order ${\boldsymbol\gamma}$. We take $\mathsf{disj}({\boldsymbol\beta},{\boldsymbol\alpha})$ as ${\boldsymbol\gamma}$. By Lemma \ref{disjunction_of_orders}  (2) we have $\mathsf{ACA}_0\vdash \mathsf{WO}_{{\boldsymbol\gamma}}\mathrel{\leftrightarrow} (\mathsf{WO}_{{\boldsymbol\alpha})}\lor \mathsf{WO}_{{\boldsymbol\beta}}).$ Thus $\mathsf{ACA}_0\vdash \mathsf{WO}_{{\boldsymbol\gamma}}\mathrel{\leftrightarrow} (\varphi \to \mathsf{WO}_{{\boldsymbol\beta}})$ and hence $U\vdash \mathsf{WO}_{{\boldsymbol\gamma}}$.  To finish the proof we note that by Lemma \ref{disjunction_of_orders}  (3)  we have $|{\boldsymbol\beta}^{\star}|\le |{\boldsymbol\gamma}^{\star}|$.\end{proof}

In a strong meta-theory for any $\mathbf{L}_2$-theory $T$ we write $|T|_{\mathbf{\Pi}^1_2}$ to denote the function $|\alpha|\longmapsto |T+\mathsf{WO}(\dot\alpha)|_{\mathbf{\Pi}^1_1}$, where $\alpha$ ranges over countable linear orders. Note that the function is well-defined from $\omega_1\cup\{\infty\}$ to $\omega_1\cup\{\infty\}$. For a function $F$ from countable linear orders to countable linear orders we denote as $|F|$ the corresponding multi-function $|\alpha|\longmapsto |F(\alpha)|$. In practice we will only consider $F$ such that $|F|$ is a function. 

In $\mathsf{ACA}_0$ for an arithmetical term $F(X)$ we write $|F|=|T|_{\mathbf{\Pi}^1_2}$ if for any linear order $\alpha$ we have $|F(\alpha)|=|T+\mathsf{WO}(\alpha)|_{\mathbf{\Pi}^1_2}$.

We often break proofs of a claim of the form $|T|_{\mathbf{\Pi}^1_2}=|F|$ into two steps. The first step, which we label $|T|_{\mathbf{\Pi}^1_2}\geq|F|$ is accomplished by proving the following claim:
$$\text{For all $\alpha$ and all $\beta<\alpha$, } T +\mathsf{WO}(\alpha) \vdash \mathsf{WO}\big(F(\beta)\big).$$
Indeed, with that claim on board it follows that for any $\alpha$, $|T+\mathsf{WO}(\alpha)|_{\mathbf{\Pi}^1_1} \geq F(\alpha)$.

We label the second step $|T|_{\mathbf{\Pi}^1_2}\leq|F|$. To execute this step we prove the following claim:
$$\text{For every $\alpha$, }\mathsf{ACA}_0\vdash \mathsf{WO}\big(F(\alpha)\big) \rightarrow \mathsf{Con}\big(T+\mathsf{WO}(\alpha)\big).$$
Indeed, with that claim on board, since $T$ always contains $\mathsf{ACA}_0$, it follows that $|T+\mathsf{WO}(\alpha)|_{\mathbf{\Pi}^1_1} \leq F(\alpha)$.

\subsection{Iterated reflection and dilators}



For an ordinal $\alpha$, we write $\varepsilon^+(\alpha)$ to denote the least $\varepsilon$-number strictly greater than $\alpha$. Formally speaking, $\varepsilon^+$ is the arithmetical term representing the naturally defined function mapping a linear order $\alpha$ to the notation system $\varepsilon^+(\alpha)$ for the least $\varepsilon$-number strictly greater than $\alpha$.

The relativization of the usual proof-theoretic analysis of $\mathsf{ACA}_0$ yields:
\begin{theorem}\label{ACA_0_Pi^1_2}
$$|\mathsf{ACA}_0|_{\mathbf{\Pi}^1_2}=|\varepsilon^+|.$$
\end{theorem}


In a standard manner for linear orders $\alpha,\beta$ we define the linear order $\phi_\alpha^+(\beta)$ that is the notation system intended for the least value of $\phi_\alpha$-function strictly above $\beta$. And for linear orders $\alpha,\beta,\gamma$ we denote as $\phi_\alpha^{+\gamma}(\beta)$ the standard notation system for the $\gamma$-th value of $\phi_\alpha$-function strictly greater than $\beta$. Formally we treat $\phi^{+}_x(y)$ and $\phi^{+z}_{x}(y)$ as binary and ternary arithmetical terms respectively representing corresponding operations on linear orders.

\begin{theorem}[$\mathsf{ACA}_0$]\label{iter_Pi^1_2}
Suppose for a $\mathbf{\Pi}^1_2$-axiomatizable theory $T$  we  have $|T|_{\mathbf{\Pi}^1_2}=|\phi_\alpha^+|$, for some linear order $\alpha$. Then for any $\beta$ we have $|\mathbf{\Pi}^1_2\mbox{-}\mathbf{R}^\beta(T)|_{\mathbf{\Pi}^1_2}=|\phi_{\alpha}^{+\omega^\beta}|$.
\end{theorem}
\begin{proof}
We reason by L\"ob's theorem over $\mathsf{ACA}_0$. That is, we reason in $\mathsf{ACA}_0$ and show that the theorem holds assuming its provability in $\mathsf{ACA}_0$. 

We consider some  $\mathbf{\Pi}^1_2$-axiomatizable theory $T$ and linear orders $\alpha,\beta,\gamma$ such that $|T|_{\mathbf{\Pi}^1_2}=|\phi^+_\alpha$|. We need to show that $|\mathbf{\Pi}^1_2\text{-}\mathbf{R}^\beta(T)+\mathsf{WO}(\gamma)|_{\mathbf{\Pi}^1_1}=|\phi_{\alpha}^{+\omega^\beta}(\gamma)|$.

For a suitably large fragment  $H$ of $\mathbf{L}$ by Lemma \ref{reduction_property} we have 
$$\mathbf{\Pi}^1_2\text{-}\mathbf{R}^\beta(T)+\mathsf{WO}(\gamma)\equiv_{\mathbf{\Pi}^1_1} T+\mathsf{WO}(\gamma)+\bigg\{\frac{\varphi\in \mathbf{\Pi}^1_1\cap H}{\mathbf{\Pi}^1_1\text{-}\mathsf{RFN}(\mathbf{\Pi}^1_2\text{-}\mathbf{R}^\delta(T)+\varphi)}\mid \delta<\beta\bigg\}.$$

By the $\mathbf{\Pi}^1_1$-completeness of well-foundedness, we can transform the rule just stated into:
$$\bigg\{\frac{\mathsf{WO}(\theta)}{\mathbf{\Pi}^1_1\text{-}\mathsf{RFN}\big(\mathbf{\Pi}^1_2\text{-}\mathbf{R}^\delta(T)+\mathsf{WO}(\theta)\big)}\mid \delta<\beta\bigg\}.$$
Recall that we are assuming the statement of the the theorem is provable in $\mathsf{ACA}_0$. Thus, 
$$\mathsf{ACA}_0\vdash |\mathbf{\Pi}^1_2\text{-}\mathbf{R}^\delta(T)+\mathsf{WO}(\theta)|_{\mathbf{\Pi}^1_1}=|\phi_{\alpha}^{+\omega^\delta}(\theta)|.$$
Whence:
$$\mathsf{ACA}_0\vdash \mathbf{\Pi}^1_1\text{-}\mathsf{RFN}\big(\mathbf{\Pi}^1_2\text{-}\mathbf{R}^\delta(T)+\mathsf{WO}(\theta)\big) \leftrightarrow \mathsf{WO}\big(\phi_{\alpha}^{+\omega^\delta}(\theta)\big).$$
Thus, we infer that the aforementioned rule is equivalent to the countable family: $$\Big\{\mathsf{WO}(\theta)\to \mathsf{WO}\big(\varphi_\alpha^{+\omega^\delta}(\theta) \big) \mid \delta<\beta \Big\}.$$
Putting this all together, 
$$\mathbf{\Pi}^1_2\text{-}\mathbf{R}^\beta(T)+\mathsf{WO}(\gamma)\equiv_{\mathbf{\Pi}^1_1} T+\mathsf{WO}(\gamma)+\Big\{\mathsf{WO}(\theta)\to \mathsf{WO}\big(\varphi_\alpha^{+\omega^\delta}(\theta) \big) \mid \delta<\beta \Big\}.$$
And the proof-theoretic ordinal of the latter theory is $\phi^{+\omega^\beta}_\alpha(\gamma)$.
\end{proof}

\begin{theorem}[$\mathsf{ACA}_0$]\label{omega_Pi^1_2} Suppose for a $\mathbf{\Pi}^1_2$-axiomatizable theory $T$  we  have $|T|_{\mathbf{\Pi}^1_2}=|\phi_\alpha^+|$, for some linear order $\alpha$. Then $|\mathbf{\Pi}^1_2\mbox{-}\omega\mathbf{R}(T)|_{\mathbf{\Pi}^1_2}=|\phi_{\alpha+1}^+|$.
\end{theorem}
\begin{proof}
We prove the claim by L\"{o}b's Theorem. So we assume the \emph{reflexive hypothesis}, i.e., that the statement of the theorem is provable in $\mathsf{ACA}_0$. The claim that $|\mathbf{\Pi}^1_2\mbox{-}\omega\mathbf{R}(T)|_{\mathbf{\Pi}^1_2}\leq|\phi_{\alpha+1}^+|$ is nearly immediately from the reflexive hypothesis. Indeed, we have supposed that $\mathsf{ACA}_0$ proves that
$$|\mathbf{\Pi}^1_2\mbox{-}\omega\mathbf{R}(T)|_{\mathbf{\Pi}^1_2}\leq|\phi_{\alpha+1}^+|.$$
This is to just to say that $\mathsf{ACA}_0$ proves that for all $\gamma$, 
$$|\mathbf{\Pi}^1_2\mbox{-}\omega\mathbf{R}(T)+\mathsf{WO}(\gamma)|_{\mathbf{\Pi}^1_1}\leq|\phi_{\alpha+1}^+(\gamma)|.$$
Whence $\mathsf{ACA}_0$ proves that for all $\gamma$,
$$\mathsf{WO}\big(\phi^+_{\alpha+1}(\gamma)\big) \rightarrow \mathsf{Con}\big(\mathbf{\Pi}^1_2\mbox{-}\omega\mathbf{R}(T)+\mathsf{WO}(\gamma)\big).$$

To establish that $|\mathbf{\Pi}^1_2\mbox{-}\omega\mathbf{R}(T)|_{\mathbf{\Pi}^1_2}\geq|\phi_{\alpha+1}^+|$, let $T$ be as in the statement of the theorem. Let $\gamma$ be an arbitrary well-ordering. Then the theory $$U:=T+\mathsf{WO}(\gamma)$$ is $\mathbf{\Pi}^1_2$-axiomatized. 

We are interested in the theory $V$, which is the closure of $U$ under the rule:
$$\mathsf{WO}(\beta) / \mathbf{\Pi}^1_2\text{-}\mathsf{RFN}\big(\mathbf{\Pi}^1_2\text{-}\mathbf{R}^\beta(U)\big)$$

$V$ has $\mathsf{WO}(\gamma)$ as an axiom, and it contains $T$. Since $|T|_{\mathbf{\Pi}^1_2}=|\phi^+_\alpha|$, it follows that, for each $\beta<\phi_\alpha^+(\gamma)$, $V$ proves $\mathbf{\Pi}^1_2\text{-}\mathsf{RFN}\big(\mathbf{\Pi}^1_2\text{-}\mathbf{R}^\beta(U)\big)$. By Theorem \ref{iter_Pi^1_2}, for all $\beta<\phi_\alpha^+(\gamma)$, $|V|_{\mathbf{\Pi}^1_2}$ is at least $|\phi^{+^{\omega^{\beta+1}}}_\alpha|$. So for all $\beta<\phi_\alpha^+(\gamma)$, for all $\delta< \phi^{+^{\omega^{\beta+1}}}_\alpha(\beta)$, $V$ proves $\mathsf{WO}(\delta)$.

By iteratively applying the argument in the previous paragraph, we conclude that $|V|_{\mathbf{\Pi}^1_1}= \phi^+_{\alpha+1}(\gamma)$.

By Lemma \ref{ON_iter_to_rule}, $V \equiv_{\mathbf{\Pi}^1_1}\mathbf{\Pi}^1_2\mbox{-}\mathbf{R}^{\mathrm{ON}}(U)$. And by Theorem \ref{omega_to_iter}, $$\mathbf{\Pi}^1_2\mbox{-}\mathbf{R}^{\mathrm{ON}}(U) \equiv \mathbf{\Pi}^1_2\mbox{-}\omega\mathbf{R}(U).$$

Thus, $|\mathbf{\Pi}^1_2\mbox{-}\omega\mathbf{R}(U)|_{\mathbf{\Pi}^1_1} = \phi^+_{\alpha+1}(\gamma)$.
Since $U:=T+\mathsf{WO}(\gamma)$ and $\gamma$ was arbitrary, this is just to say that $|\mathbf{\Pi}^1_2\mbox{-}\omega\mathbf{R}(T)|_{\mathbf{\Pi}^1_2}\geq|\phi_{\alpha+1}^+|$. This completes the proof.
\end{proof}

\begin{theorem}[$\mathsf{ACA}_0$]\label{omega_iter_Pi^1_2}
For any linear order $\alpha$
 $$|\mathbf{\Pi}^1_2\mbox{-}\omega\mathbf{R}^\alpha(\mathsf{ACA}_0)|_{\mathbf{\Pi}^1_2}=|\phi_{1+\alpha}^+|.$$
\end{theorem}
\begin{proof} We reason by L\"ob's theorem. We work inside $\mathsf{ACA}_0$ and assume that $\mathsf{ACA}_0$ proves the statement of the theorem.

Now fix some $\alpha$. We have assumed that $\mathsf{ACA_0}$ proves that for all $\beta<\alpha$
\begin{equation}\label{dilator_inside}
|\mathbf{\Pi}^1_2\text{-}\omega\mathbf{R}^\beta(\mathsf{ACA}_0)|_{\mathbf{\Pi}^1_2}=|\phi^+_{1+\beta}|.
\end{equation}
Note that $\mathbf{\Pi}^1_2\text{-}\omega\mathbf{R}^\beta(\mathsf{ACA}_0)$ is $\mathbf{\Pi}^1_2$-axiomatized. So by combining Theorem \ref{omega_Pi^1_2} and assumption \ref{dilator_inside}, we infer that $\mathsf{ACA}_0$ proves that for any $\beta<\alpha$ 
\begin{equation}\label{application_theory}
|\mathbf{\Pi}^1_2\text{-}\omega\mathbf{R}^1\big(\mathbf{\Pi}^1_2\text{-}\omega\mathbf{R}^\beta(\mathsf{ACA}_0)\big)|_{\mathbf{\Pi}^1_2} = |\phi^+_{1+\beta+1}|.
\end{equation}

To finish the proof we fix an arbitrary well-ordering $\gamma$ and claim that \begin{equation}\label{it_omega_on_wo}|\mathbf{\Pi}^1_2\mbox{-}\omega\mathbf{R}^\alpha(\mathsf{ACA}_0)+\mathsf{WO}(\gamma)|_{\mathbf{\Pi}^1_1}=|\phi_{1+\alpha}^+(\gamma)|.\end{equation}
We have
$$\mathbf{\Pi}^1_2\mbox{-}\omega\mathbf{R}^\alpha(\mathsf{ACA}_0)+\mathsf{WO}(\gamma)=\mathsf{ACA}_0+\mathsf{WO}(\gamma)+\{\mathbf{\Pi}^1_2\mbox{-}\omega\mathsf{RFN}(\mathbf{\Pi}^1_2\mbox{-}\omega\mathbf{R}^\beta(\mathsf{ACA}_0))\mid \beta<\alpha\}.$$
Notice that any finite fragment of $$\mathsf{ACA}_0+\mathsf{WO}(\gamma)+\{\mathbf{\Pi}^1_2\mbox{-}\omega\mathsf{RFN}(\mathbf{\Pi}^1_2\mbox{-}\omega\mathbf{R}^\beta(\mathsf{ACA}_0))\mid \beta<\alpha\}$$ is $\sqsupseteq^{\mathbf{\Sigma}^1_1}$-contained in either $\mathsf{ACA}_0+\mathsf{WO}(\gamma)$ or $\mathbf{\Pi}^1_2\text{-}\omega\mathbf{R}^1(\mathbf{\Pi}^1_2\text{-}\omega\mathbf{R}^\beta(\mathsf{ACA}_0))+\mathsf{WO}(\gamma)$, for some $\beta<\alpha$. Thus:
$$|\mathbf{\Pi}^1_2\mbox{-}\omega\mathbf{R}^\alpha(\mathsf{ACA}_0)+\mathsf{WO}(\gamma)|_{\mathbf{\Pi}^1_1}=\max(\varepsilon^+(\gamma),\sup\limits_{\beta<\alpha}\phi_{1+\beta+1}^+(\gamma))=\phi_{1+\alpha}^+(\gamma).$$
This completes the proof of the theorem.
\end{proof}

Let $\Gamma^+(\alpha)$ be the notation system for the least $\Gamma$-number strictly greater than $\alpha$. Once again, we formally treat $\Gamma^+$ as the arithmetical term representing the corresponding operation on linear orders.

\begin{theorem}\label{ATR_0_Pi^1_2}
$|\mathsf{ATR}_0|_{\mathbf{\Pi}^1_2}=|\Gamma^+|$.
\end{theorem}
\begin{proof} We consider the theory $T:=\mathsf{ATR}_0+\mathsf{WO}(\gamma)$.
By Corollary \ref{reflection_ATR_0}, we may put $T$ into the form:
$$\mathsf{WO}(\gamma) + \forall \alpha \Big(\mathsf{WO}(\alpha) \rightarrow \mathbf{\Pi}^1_2\textrm{-}\omega\mathsf{RFN}\big(\mathbf{\Pi}^1_2\textrm{-}\omega\mathbf{R}^\alpha(\mathsf{ACA}_0)\big)\Big).$$ 
Then $T$ clearly proves
$$ \mathbf{\Pi}^1_2\textrm{-}\omega\mathsf{RFN}\big(\mathbf{\Pi}^1_2\textrm{-}\omega\mathbf{R}^\gamma(\mathsf{ACA}_0)\big).$$
By Lemma \ref{ON_iter_to_rule}, $|T|_{\mathbf{\Pi}^1_1}\geq \phi^+_{1+\alpha}(\gamma).$ So for every $\delta< \phi^+_{1+\alpha}(\gamma)$, $T$ proves
$$ \mathbf{\Pi}^1_2\textrm{-}\omega\mathsf{RFN}\big(\mathbf{\Pi}^1_2\textrm{-}\omega\mathbf{R}^\delta(\mathsf{ACA}_0)\big).$$
Whence by Lemma \ref{ON_iter_to_rule} again, $|T|_{\mathbf{\Pi}^1_1}\geq \phi^+_{1+{\phi^+_{1+\alpha}(\gamma)}}(\gamma).$

By iterating this argument, we see that $|T|_{\mathbf{\Pi}^1_1}\geq \Gamma^+(\gamma)$.  That is, $|\mathsf{ATR}_0|_{\mathbf{\Pi}^1_2}\geq\Gamma^+$.

To see that $|\mathsf{ATR}_0|_{\mathbf{\Pi}^1_2}\leq\Gamma^+$ we need to show that: 
$$\mathsf{ACA}_0 \vdash \mathsf{WF}\big(\Gamma^+ (\alpha) \big) \to\mathbf{\Pi}^1_1\text{-}\mathsf{RFN}\big(\mathsf{ATR}_0 + \mathsf{WO}(\alpha)\big).$$

Let's reason in $\mathsf{ACA}_0$ and make some observations about the claim $\mathbf{\Pi}^1_1\text{-}\mathsf{RFN}\big(\mathsf{ATR}_0 + \mathsf{WO}(\alpha)\big)$.
First, by Corollary \ref{reflection_ATR_0}, $\mathsf{ATR}_0+\mathsf{WO}(\alpha)$ is equivalent to:
$$\mathsf{WO}(\alpha) + \forall\gamma \Big( \mathsf{WO}(\gamma) \to \mathbf{\Pi}^1_2\text{-}\omega\mathsf{RFN}\big(\mathbf{\Pi}^1_2\text{-}\omega\mathbf{R}^\gamma(\mathsf{ACA}_0)\big)\Big)$$
By Theorem \ref{ON_omega_iter_to_rule} this is $\mathbf{\Pi}^1_1$-equivalent to the closure of $\mathsf{ACA}_0+\mathsf{WO}(\alpha)$ under the rule:
$$\frac{\mathsf{WO}(\gamma)}{\mathbf{\Pi}^1_2\mbox{-}\omega\mathsf{RFN}\Big(\mathbf{\Pi}^1_2\mbox{-}\omega\mathbf{R}^{\gamma}\big(\mathsf{ACA}_0 + \mathsf{WO}(\alpha)\big)\Big)}$$
Using Theorem \ref{omega_iter_Pi^1_2} we can bound the provably well-founded ordinals of nested applications of this rule. Indeed, the upper-bounds on the ordinals that are provably well-founded by nested applications of the rule are:
$$\{ \phi^+_1(\alpha), \phi^+_{1+\phi^+_1(\alpha)}(\alpha), \phi^+_{1+\phi^+_{1+\phi^+_1(\alpha)}(\alpha)}(\alpha) \dots \}$$
the supremum of which is $\Gamma^+(\alpha)$. By the $\mathbf{\Pi}^1_1$-completeness of well-foundedness, every $\mathbf{\Pi}^1_1$ theorem of $\mathsf{ATR}_0+\mathsf{WO}(\alpha)$ follows from the claim that $\mathsf{WO}(\beta)$ for some $\beta<\Gamma^+(\alpha)$. 

Since all of that reasoning took place in $\mathsf{ACA}_0$, we conclude that: $$\mathsf{ACA}_0 + \mathsf{WO}\big(\Gamma^+(\alpha)\big)\vdash \mathbf{\Pi}^1_1\text{-}\mathsf{RFN}\big(\mathsf{ATR}_0+\mathsf{WO}(\alpha)\big).$$
This completes the proof of the theorem.
\end{proof}


\section{$\Sigma^1_1\mbox{-}\mathsf{AC}_0$, $\Sigma^1_1\mbox{-}\mathsf{DC}_0$ and equivalents of $\mathsf{ATR}_0$}\label{choice}

The main theorem of this section is an alternative axiomatization of $\mathsf{ATR}_0$ in terms of reflection principles. First we give a new proof of the fact that $\mathsf{ATR}_0$ is equivalent to the claim ``every set is contained in an $\omega$-model of $\Sigma^1_1\text{-}\mathsf{AC}_0$'' (see \cite[Lemma~VIII.4.19]{simpson2009subsystems}). We then prove that $\mathsf{ATR}_0$ is equivalent to the statement ``every set is contained in an $\omega$-model of $\mathbf{\Pi}^1_3\text{-}\omega\mathsf{RFN}(\mathsf{ACA}_0)$.''
To the best of our knowledge, this exact characterization has not appeared in the literature. However, one may indirectly prove this equivalence by combining two known results. First, a proof that $\mathsf{ATR}_0$ is equivalent to ``every set is contained in an $\omega$-model of $\mathsf{DC}$'' is given by Avigad and Sommer \cite{avigad1999model}. Second, a proof of the equivalence of  $\mathbf{\Pi}^1_3\text{-}\omega\mathsf{RFN}(\mathsf{ACA}_0)$ and $\mathsf{DC}$ appears in Simpson's book \cite[Theorem~VIII.5.12]{simpson2009subsystems}.

We say that a sentence is $\forall\exists_2! \mathbf{\Pi}^0_\infty$ over a theory $T$ if it is of the form $\forall X\exists Y\varphi(X,Y)$, where $\varphi \in\mathbf{\Pi}^0_\infty$ and $T\vdash \forall X,Y_1,Y_2(\varphi(X,Y_1)\land \varphi(X,Y_2)\to Y_1=Y_2)$. We say that a theory $T$ is $\forall\exists_2! \mathbf{\Pi}^0_\infty$-axiomatizable over $\mathbf{\Sigma}^1_1\text{-}\mathsf{Prv}_{\mathsf{ACA}_0}$ if there exists a true $\mathbf{\Sigma}^1_1$-sentence $\varphi$ so that for any axiom $\psi$ of $T$ there is a $\forall X\exists_2! \mathbf{\Pi}^0_\infty$ sentence $\psi'$ over $T+\varphi$ so that $\mathsf{ACA}_0\vdash \varphi\land \psi'\to \psi$ and $T\vdash \varphi\to \psi'$.

\begin{lemma}[$\mathsf{ACA}_0$]\label{AC_cons_1}
For any $\forall\exists_2! \mathbf{\Pi}^0_\infty$-axiomatizable over $\mathbf{\Sigma}^1_1\text{-}\mathsf{Prv}_{\mathsf{ACA}_0}$ theory $T$ we have $$\mathsf{ACA}_0+\mathbf{\Pi}^1_2\textrm{-}\omega\mathsf{RFN}(T)\subseteq^{\mathbf{\Sigma}^1_1} \Sigma^1_1\textrm{-}\mathsf{AC}_0+\mathbf{\Pi}^1_2\textrm{-}\mathsf{RFN}(T).$$
\end{lemma}
\begin{proof} 
We reason in $\mathbf{\Pi}^1_2\text{-}\mathsf{RFN}_{\mathsf{ACA}_0}(T)+\Sigma^1_1\text{-}\mathsf{AC}_0+$ ``$T$ is $\forall\exists_2! \mathbf{\Pi}^0_\infty$-axiomatizable over $\mathbf{\Sigma}^1_1\text{-}\mathsf{Prv}_{\mathsf{ACA}_0}$'' and are going to prove $\mathbf{\Pi}^1_2\textrm{-}\omega\mathsf{RFN}(T)$. 

Since $T$ is $\forall\exists_2! \mathbf{\Pi}^0_\infty$-axiomatizable over $\mathbf{\Sigma}^1_1\text{-}\mathsf{Prv}_{\mathsf{ACA}_0}$ we could fix a true $
\mathbf{\Sigma}^1_1$-sentence $\varphi$ such that for any axiom $\psi$ of $T$ there is a sentence $\psi'$ so that $\mathsf{ACA}_0\vdash \varphi\land \psi'\to \psi$ and $T\vdash \varphi\to\psi'$. Consider an arbitrary set $A$. It will be enough to show that there is a countably coded $\omega$-model $\mathfrak{M}_A$ of $T$ that contains $A$. Let $H$ be the countable fragment of $\mathbf{L}_2$ that contains all set constants used in the axioms of $T$ and in the sentence $\varphi$ and also contains the constant for the set $A$. Let $\psi_0(X),\psi_1(X),\ldots$ be some fixed enumeration of $H$-formulas of the complexity $\mathbf{\Sigma}^1_1$ without other free variables such that $T+\varphi\vdash \exists! X \psi_i(X)).$ By $\mathbf{\Pi}^1_2\text{-}\mathsf{RFN}(T)$ we have that $\forall i \exists! X \mathsf{Tr}_{\mathbf{\Sigma}^1_1}(\psi_i(X))$. Hence by $\Sigma^1_1\text{-}\mathsf{AC}_0$ there exists a set $S$ such that $\mathsf{Tr}_{\mathbf{\Sigma}^1_1}(\psi_i((S)_i))$. We claim that $S$ as a countable collection of sets is an $\omega$-model of $T$ containing $A$. 

Since $X=A$ is $\psi_i(X)$ for some $i$, the collection $S$ should contain $A$ as $(S)_i$ for this particular $i$. By the same argument we see that $S$ contains witnesses for $\varphi$ and hence $S$ satisfies $\varphi$. It is now enough to check that $S$ as an $\omega$-model satisfies all instances of arithmetical comprehension and all $\forall \exists_2!\mathbf{\Pi}^0_\infty$-consequences of $T$. 

Let us first prove that $S$ satisfies an instance of arithmetical comprehension $$\forall X_1,\ldots, X_n \forall x_1,\ldots,x_m \exists Y\forall y(y\in Y\mathrel{\leftrightarrow}\theta(y,X_1,\ldots,X_n,x_1,\ldots,x_m)).$$ We consider sets $(S)_{i_1},\ldots,(S)_{i_n}$ and numbers $a_1,\ldots,a_m$. We claim that $$\exists Y\forall y(y\in Y\mathrel{\leftrightarrow}\theta(y,(S)_{i_1},\ldots,(S)_{i_n},a_1,\ldots,a_m))$$ holds in $S$. Indeed, $$T\vdash \exists ! Y\; \exists X_1,\ldots,X_n\Big(\psi_{i_1}(X_1)\land\ldots\land \psi_{i_n}(X_n)\land \forall y\big(y\in Y\mathrel{\leftrightarrow}\theta(y,X_1,\ldots,X_n,a_1,\ldots,a_m)\big)\Big)$$
which allows us to find the result of arithmetical comprehension in $S$. 

Let us consider a sentence $\forall X\exists Y \theta(X,Y)$ that is a $\forall \exists_2!\mathbf{\Pi}^0_\infty$-consequences of $T$ and verify that this sentence holds in $S$. We consider an arbitrary $(S)_{i}$ and claim that there is $j$ such that $\theta((S)_{i},(S)_{j})$ holds. Indeed we choose $j$ so that $\psi_j(X)$ is the result of pushing existential quantifiers to the front of $\exists Y(\psi_j(Y)\land \theta(Y,X))$.\end{proof}

\begin{lemma}[$\mathsf{ACA}_0$]\label{AC_cons_AEU} For any $\forall \exists_2!\mathbf{\Pi}^0_\infty$-axiomatizable over $\mathbf{\Sigma}^1_1\textrm{-}\mathsf{Prv}_{\mathsf{ACA}_0}$ theory $T$ the theory $\mathsf{ACA}_0+\mathbf{\Pi}^1_2\textrm{-}\omega\mathsf{RFN}(T)$ is also $\forall \exists_2!\mathbf{\Pi}^0_\infty$-axiomatizable over $\mathbf{\Sigma}^1_1\textrm{-}\mathsf{Prv}_{\mathsf{ACA}_0}$.\end{lemma}
\begin{proof}
  Our goal will be to put $\mathbf{\Pi}^1_2\textrm{-}\omega\mathsf{RFN}(T)$ in  $\forall \exists_2!\mathbf{\Pi}^0_\infty$-form working over $\mathbf{\Sigma}^1_1\textrm{-}\mathsf{Prv}_{\mathsf{ACA}_0}$. Observe that the model $\mathfrak{M}_A$ constructed in the proof of Lemma \ref{AC_cons_1} is definable by a $\mathbf{\Pi}^0_\infty$ formula (with $A$ as parameter). Furthermore, the model $\mathfrak{M}_A$ could be constructed in $\mathsf{ACA}_0$ from any $\omega$-model of $T$ containing $A$ (rather than using $\Sigma_1^1\textrm{-}\mathsf{AC}_0$ as we have done in Lemma \ref{AC_cons_1}).
\end{proof}

\begin{lemma}[$\mathsf{ACA}_0$]\label{AC_cons_2}
For any $\forall\exists_2! \mathbf{\Pi}^0_\infty$-axiomatizable over $\mathbf{\Sigma}^1_1\text{-}\mathsf{Prv}_{\mathsf{ACA}_0}$ theory $T$ we have $$\mathsf{ACA}_0+\mathbf{\Pi}^1_2\textrm{-}\omega\mathsf{RFN}(T)\supseteq_{\mathbf{\Pi}^1_2}^{\mathbf{\Sigma}^1_1} \Sigma^1_1\textrm{-}\mathsf{AC}_0+\mathbf{\Pi}^1_2\textrm{-}\mathsf{RFN}(T).$$
\end{lemma}
\begin{proof}
 Our proof is inspired by the proof of \cite[Theorem~IX.4.4]{simpson2009subsystems} by Simpson. It is enough for us to consider an aribitrary true $\mathbf{\Sigma}^1_1$-sentence $\varphi$ and arbitrary $\mathbf{\Sigma}^1_2$-sentence $\psi$ such that the theory $U=\mathbf{\Pi}^1_2\textrm{-}\omega\mathsf{RFN}(T)+\varphi+\psi$ is consistent and show that $V=\Sigma^1_1\textrm{-}\mathsf{AC}_0+\mathbf{\Pi}^1_2\textrm{-}\mathsf{RFN}(T)+\varphi+\psi$ is consistent as well. 
 
 Indeed, we consider a countable fragment $H$ of $\mathbf{L}_2$ covering all the constants used in  theories $U$ and $V$. We further extend the language $H$ by a family $C_0,C_1,\ldots$ of fresh constants of set type (these constants are not from $\mathbf{L}_2$ and hence do not correspond to any particular set). We denote the resulting language $H'$. We denote by $U'$ the $H'$ theory that extends $U$ by the axioms 
\begin{enumerate}
    \item ``$C_{i}$ is an $\omega$-model of $T$'', for all $i$;
    \item $C_{i}\in C_{i+1}$, for all $i$; 
    \item  ``$C_0$ contains witnesses for $\varphi$ and $\psi$'';
    \item  $A\in C_0$, for all constants $A$ from $H$. 
\end{enumerate}
Since any finite fragment of $U'$ could be interpreted in $U$, we see that $U'$ is consistent. By the recursively saturated models existence theorem (it is provable in $\mathsf{WKL}_0$ \cite[Lemma~IX.4.2]{simpson2009subsystems}) there is a model $\mathfrak{M}$ of $U'$ that is $H$-recursively saturated. We define an $H$-model $\mathfrak{N}$ that has the same first-order part as $\mathfrak{M}$, but its second-order part is restricted to $\mathfrak{M}$-sets $X$ such that $X\in C_i$, for some $i$. Clearly the model $\mathfrak{N}$ satisfies $U$. To finish the proof we show that $\mathfrak{N}$ satisfies $\Sigma^1_1\textrm{-}\mathsf{AC}_0$ and hence is the model of $V$. 

Indeed, we consider some $\Sigma^1_1$-formula $\varphi(x,Y)$ with parameters from $\mathfrak{N}$ such that $\mathfrak{N}\models \forall x\exists Y\;\varphi(x,Y)$ and claim that there is an $\mathfrak{N}$-set $A$ such that $\mathfrak{N}\models \forall x\;\varphi(x,(A)_x)$. Assume, for the sake of contradiction, that no $A$ with this property exists. Observe that in this case for any natural $i$ we would have $\mathfrak{M}\models \exists x\; \forall y\;\lnot \varphi(x,(C_i)_y)$. Hence by $H$-recursive saturation of $\mathfrak{M}$ there should be a non-standard number $a$ such that $\mathfrak{M}\models \forall y\;\lnot \varphi(a,(C_i)_y)$, for any $i$. But this means that $\mathfrak{N}\models \lnot\exists Y\;\varphi(a,Y)$, contradiction. \end{proof}

\begin{corollary}[$\mathsf{ACA}_0$]\label{AC_cons_iter}
For any $\forall\exists_2! \mathbf{\Pi}^0_\infty$-axiomatizable theory $T$ we have $$\mathbf{\Pi}^1_2\textrm{-}\omega\mathbf{R}^\alpha(T)\equiv^{\mathbf{\Sigma}^1_1}_{\mathbf{\Pi}^1_2} \mathbf{\Pi}^1_2\textrm{-}\mathbf{R}^\alpha(T+\Sigma^1_1\textrm{-}\mathsf{AC}_0).$$
\end{corollary}
\begin{proof}  We strengthen the corollary by the assertion that $\mathbf{\Pi}^1_2\textrm{-}\omega\mathbf{R}^\alpha(T)$ is $\forall\exists_2!\mathbf{\Pi}^0_\infty$-axiomatizable over $\mathbf{\Sigma}^1_1\textrm{-}\mathsf{Prv}_{\mathsf{ACA}_0}$. 
We use L\"{o}b's Theorem to prove the strengthen version of the corollary.  Reason in $\mathsf{ACA}_0$ and assume that corollary is $\mathsf{ACA}_0$-provable.


Consider an order $\alpha$. \textbf{First observation}: Any $\mathbf{\Pi}^1_2$-theorem $\varphi$ of $\mathbf{\Pi}^1_2\textrm{-}\omega\mathbf{R}^\alpha(T)$ is a theorem of  $$\mathsf{ACA}_0+\mathbf{\Pi}^1_2\textrm{-}\omega\mathsf{RFN}\big(\mathbf{\Pi}^1_2\textrm{-}\omega\mathbf{R}^\beta(T)\big)$$ for some suborder $\beta$ of $\alpha$. 

\textbf{Second observation}: Any $\mathbf{\Pi}^1_2$-theorem $\varphi$ of $\mathbf{\Pi}^1_2\textrm{-}\mathbf{R}^\alpha(T+\Sigma^1_1\textrm{-}\mathsf{AC}_0)$ is a theorem of $$\Sigma^1_1\textrm{-}\mathsf{AC}_0+\mathbf{\Pi}^1_2\textrm{-}\mathsf{RFN}\big(\mathbf{\Pi}^1_2\textrm{-}\mathbf{R}^\beta(T+\Sigma^1_1\textrm{-}\mathsf{AC}_0)\big)$$ for some suborder $\beta$ of $\alpha$.

\textbf{Third observation:} By the provability of the present corollary the theory $$\Sigma^1_1\textrm{-}\mathsf{AC}_0+\mathbf{\Pi}^1_2\textrm{-}\mathsf{RFN}(\mathbf{\Pi}^1_2\textrm{-}\mathbf{R}^\beta\big(T+\Sigma^1_1\textrm{-}\mathsf{AC}_0)\big)$$ is $\mathbf{\Sigma}^1_1\textrm{-}\mathsf{Prv}_{\mathsf{ACA}_0}$-equivalent to $\Sigma^1_1\textrm{-}\mathsf{AC}_0+\mathbf{\Pi}^1_2\textrm{-}\mathsf{RFN}\big(\mathbf{\Pi}^1_2\textrm{-}\omega\mathbf{R}^\beta(T)\big)$. 

Accordingly, we infer from from the second and third observations the \textbf{fourth observation}: Any $\mathbf{\Pi}^1_2$-theorem $\varphi$ of $\mathbf{\Pi}^1_2\textrm{-}\mathbf{R}^\alpha(T+\Sigma^1_1\textrm{-}\mathsf{AC}_0)$ is $\Sigma^1_1$ provable in $$\Sigma^1_1\textrm{-}\mathsf{AC}_0+\mathbf{\Pi}^1_2\textrm{-}\mathsf{RFN}\big(\mathbf{\Pi}^1_2\textrm{-}\omega\mathbf{R}^\beta(T)\big)$$ for some suborder $\beta$ of $\alpha$.

From the first observation and Lemma \ref{AC_cons_1} we infer that:
\begin{equation}\label{fourth-conservation}
    \mathbf{\Pi}^1_2\text{-}\mathbf{R}^\alpha(T+\Sigma^1_1\text{-}\mathsf{AC}_0+\varphi) \subseteq^{\Sigma^1_1}_{\mathbf{\Pi}^1_2} \Sigma^1_1\text{-}\mathsf{AC}_0 + \mathbf{\Pi}^1_2\text{-}\mathsf{RFN}(T)
\end{equation}
From the fourth observation and Lemma \ref{AC_cons_2} we infer that:
\begin{equation} \label{fifth-conservation} 
\mathbf{\Pi}^1_2\text{-}\mathbf{R}^\alpha(T + \Sigma^1_1\text{-}\mathsf{AC}_0 +\varphi) \subseteq^{\Sigma^1_1}_{\mathbf{\Pi}^1_2} \mathsf{ACA}_0+\mathbf{\Pi}^1_2\text{-}\omega\mathsf{RFN}\big(\mathbf{\Pi}\text{-}\omega\mathbf{R}^\beta(T)\big)
\end{equation}
Applying Lemma \ref{AC_cons_AEU} to (\ref{fourth-conservation}) we see that Lemma \ref{AC_cons_AEU} is applicable to (\ref{fifth-conservation}). And the corollary follows by applying Lemma \ref{AC_cons_AEU} to (\ref{fifth-conservation}).
\end{proof}

\begin{proposition}\label{ATR_0_Sigma^1_1_AC}
$$\mathsf{ATR}_0\equiv \mathsf{ACA}_0+\mathbf{\Pi}^1_2\mbox{-}\omega\mathsf{RFN}(\Sigma^1_1\mbox{-}\mathsf{AC}_0)$$
\end{proposition}
\begin{proof}
We reason in $\mathsf{ACA}_0$. \begin{flalign*}
\mathbf{\Pi}^1_2\mbox{-}\omega\mathsf{RFN}(\Sigma^1_1\mbox{-}\mathsf{AC}_0) & \equiv \forall \alpha\Big(\mathsf{WO}(\alpha)\to \mathbf{\Pi}^1_2\mbox{-}\mathsf{RFN}\big(\mathbf{\Pi}^1_2\mbox{-}\mathbf{R}^\alpha(\Sigma^1_1\mbox{-}\mathsf{AC}_0)\big)\Big) \text{ by Theorem \ref{omega_to_iter}}\\
&\equiv \forall  \alpha\Big(\mathsf{WO}(\alpha)\to \mathbf{\Pi}^1_2\mbox{-}\omega\mathsf{RFN}\big(\mathbf{\Pi}^1_2\mbox{-}\omega\mathbf{R}^\alpha(\mathsf{ACA}_0)\big)\Big) \text{ by Corollary \ref{AC_cons_iter}}\\
& \equiv \forall \alpha \forall X \big( \mathsf{WO}(\alpha) \to(X\sqcup \alpha)^{(\omega^{1+\alpha})}\text{ exists}\big) \text{ by Theorem \ref{omegaRFN_Turing} }\\
&\equiv \mathsf{ATR}_0
\end{flalign*}
Note that the last equivalence is well-known.
\end{proof}


Note that over $\mathsf{ACA}_0$ the principles $\Sigma^1_1\mbox{-}\mathsf{DC}, \Pi^1_1\mbox{-}\mathsf{BI},$ and $\mathbf{\Pi}^1_3\mbox{-}\omega\mathsf{RFN}$ are equivalent \cite[Theorem~VIII.5.12]{simpson2009subsystems}.

\begin{lemma}[$\mathsf{ACA}_0$]\label{DC_iter}
$$\mathbf{\Pi}^1_2\mbox{-}\mathbf{R}^\alpha\big(\mathsf{ACA}_0+ \mathbf{\Pi}^1_3\mbox{-}\omega\mathsf{RFN}(\mathsf{ACA}_0)\big)\equiv^{\mathbf{\Sigma}^1_1}_{\mathbf{\Pi}^1_2}\mathbf{\Pi}^1_2\mbox{-}\omega\mathbf{R}^{\omega\alpha}(\mathsf{ACA}_0)$$
\end{lemma}
\begin{proof}
We prove the lemma using L\"{o}b's Theorem. We reason in $\mathsf{ACA}_0$. We assume that the lemma is provable in $\mathsf{ACA}_0$ and claim that the lemma holds. We consider some particular linear order $\alpha$ and need to show that 
$$U=\mathbf{\Pi}^1_2\mbox{-}\mathbf{R}^\alpha\big(\mathsf{ACA}_0+ \mathbf{\Pi}^1_3\mbox{-}\omega\mathsf{RFN}(\mathsf{ACA}_0)\big)\equiv^{\mathbf{\Sigma}^1_1}_{\mathbf{\Pi}^1_2}\mathbf{\Pi}^1_2\mbox{-}\omega\mathbf{R}^{\omega\alpha}(\mathsf{ACA}_0)=V.$$

We start with proving $U\supseteq^{\mathbf{\Sigma}^1_1}_{\mathbf{\Pi}^1_2} V$. It is enough to show that for any $\beta\prec \alpha$ and $n<\omega$, the sentence $\mathbf{\Pi}^1_2\textrm{-}\mathsf{RFN}(\mathbf{\Pi}^1_2\mbox{-}\mathbf{R}^{\omega\beta+n}\big(\mathsf{ACA}_0)\big)$ is provable in $U$.  By the provability of the present lemma we have the provability of the equivalence
$$\mathbf{\Pi}^1_2\textrm{-}\mathsf{RFN}\Big(\mathbf{\Pi}^1_2\mbox{-}\mathbf{R}^\beta\big(\mathsf{ACA}_0+ \mathbf{\Pi}^1_3\mbox{-}\omega\mathsf{RFN}(\mathsf{ACA}_0)\big)\Big)\mathrel{\leftrightarrow} \mathbf{\Pi}^1_2\textrm{-}\mathsf{RFN}\big(\mathbf{\Pi}^1_2\mbox{-}\mathbf{R}^{\omega\beta}(\mathsf{ACA}_0)\big).$$
Thus $U$ proves $\mathbf{\Pi}^1_2\textrm{-}\mathsf{RFN}(\mathbf{\Pi}^1_2\mbox{-}\mathbf{R}^{\omega\beta}\big(\mathsf{ACA}_0)\big)$. Next by induction (using closure of $V$ under $\mathbf{\Pi}^1_2\textrm{-}\mathsf{RFN}$ rule) we establish that $U$ proves $\mathbf{\Pi}^1_2\textrm{-}\mathsf{RFN}(\mathbf{\Pi}^1_2\mbox{-}\mathbf{R}^{\omega\beta+m}\big(\mathsf{ACA}_0)\big)$, for all natural $m$.

Now let us prove that $U\subseteq^{\mathbf{\Sigma}^1_1}_{\mathbf{\Pi}^1_2} V$. We need to show that for any  $\beta<\alpha$ we have $$U'=\mathsf{ACA}_0+ \mathbf{\Pi}^1_3\mbox{-}\omega\mathsf{RFN}(\mathsf{ACA}_0)+\mathbf{\Pi}^1_2\mbox{-}\omega\mathsf{RFN}\Big(\mathbf{\Pi}^1_2\mbox{-}\mathbf{R}^\beta\big(\mathsf{ACA}_0+ \mathbf{\Pi}^1_3\mbox{-}\omega\mathsf{RFN}(\mathsf{ACA}_0)\big)\Big)\supseteq^{\mathbf{\Sigma}^1_1}_{\mathbf{\Pi}^1_2} V.$$
Indeed, the provability of the lemma implies that $V$ proves $$\mathbf{\Pi}^1_2\mbox{-}\omega\mathsf{RFN}\Big(\mathbf{\Pi}^1_2\mbox{-}\mathbf{R}^\beta\big(\mathsf{ACA}_0+ \mathbf{\Pi}^1_3\mbox{-}\omega\mathsf{RFN}(\mathsf{ACA}_0)\big)\Big).$$ Using the closure of $V$ under the $\mathbf{\Pi}^1_2\textrm{-}\mathsf{RFN}$ rule and Lemma \ref{reduction_property} we get the desired  inclusion $U'\subseteq^{\mathbf{\Sigma}^1_1}_{\mathbf{\Pi}^1_2} V$.\end{proof}

\begin{theorem}
$$\mathsf{ATR}_0\equiv\mathsf{ACA}_0+\mathbf{\Pi}^1_2\mbox{-}\omega\mathsf{RFN}(\mathsf{ACA}_0+ \mathbf{\Pi}^1_3\mbox{-}\omega\mathsf{RFN}\big(\mathsf{ACA}_0)\big)$$
\end{theorem} 
\begin{proof}
The following sequence of pairwise equivalent sentences establishes the equivalence:
\begin{flalign*}
\mathsf{ATR}_0 &\equiv \forall \alpha\Big(\mathsf{WO}(\alpha)\to \mathbf{\Pi}^1_2\mbox{-}\omega\mathsf{RFN}\big(\mathbf{\Pi}^1_2\mbox{-}\omega\mathbf{R}^\alpha(\mathsf{ACA}_0)\big)\Big) \text{ by Corollary \ref{reflection_ATR_0} }\\
& \equiv \forall \alpha\Big(\mathsf{WO}(\alpha)\to \mathbf{\Pi}^1_2\mbox{-}\mathsf{RFN}\big(\mathbf{\Pi}^1_2\mbox{-}\omega\mathbf{R}^{\omega\alpha}(\mathsf{ACA}_0)\big)\Big) \text{ trivially}\\
& \equiv\forall \alpha\Big(\mathsf{WO}(\alpha)\to \mathbf{\Pi}^1_2\mbox{-}\mathsf{RFN}\big(\mathbf{\Pi}^1_2\mbox{-}\mathbf{R}^\alpha(\mathsf{ACA}_0+ \mathbf{\Pi}^1_3\mbox{-}\omega\mathsf{RFN}(\mathsf{ACA}_0))\big)\Big) \text{ by Lemma \ref{DC_iter}}\\
& \equiv \mathbf{\Pi}^1_2\mbox{-}\omega\mathsf{RFN}\big(\mathsf{ACA}_0+ \mathbf{\Pi}^1_3\mbox{-}\omega\mathsf{RFN}(\mathsf{ACA}_0)\big) \text{ by Theorem \ref{omega_to_iter}}
\end{flalign*}
This completes the proof of the theorem.
\end{proof}

\section{Proof-theoretic ordinals and well-ordering principles}\label{well-ordering-principles}

In this section we deliver proof-theoretic applications of the results from the previous sections. In particular, we use the systematic connection between $\omega$-model reflection and dilators to provide uniform proof-theoretic analyses of theories in the interval $[\mathsf{ACA}_0,\mathsf{ATR}]$. Note that these techniques are useful not only for calculating proof-theoretic ordinals but also for proving reverse-mathematical characterizations of well-ordering principles.

The following is a variant of the usual Schmerl's formula that could be proved in a standard manner \cite{beklemishev2003proof,schmerl1979fine} using Lemma \ref{reduction_property}:
\begin{theorem}\label{Schmerl-variant}For any $\mathbf{\Pi}^1_{n+1}$-axiomatizable theory $T$ we have
$$T+\mathbf{\Pi}^1_\infty\text{-}\mathsf{RFN}(T)\equiv^{\mathbf{\Sigma}^1_1}_{\mathbf{\Pi}^1_n}\mathbf{\Pi}^1_n\text{-}\mathbf{R}^{\varepsilon_0}(T).$$\end{theorem}

\begin{corollary}
\begin{enumerate}
    \item $|\mathsf{ACA}_0^+|_{\mathbf{\Pi}^1_1}=\phi_2(0)$;
    \item $|\Sigma^1_1$-$\mathsf{AC}|_{\mathbf{\Pi}^1_1}=|\Pi^1_2$-$\mathsf{RFN}^{\varepsilon_0}(\Sigma^1_1$-$\mathsf{AC}_0)|_{\mathbf{\Pi}^1_1}=\phi_{\varepsilon_0}(0)$;
    \item $|\mathsf{ATR}_0|_{\mathbf{\Pi}^1_1}=\Gamma_0$;
    \item $|\mathsf{ATR}|_{\mathbf{\Pi}^1_1}=\Gamma_{\varepsilon_0}$.
\end{enumerate}
\end{corollary}
\begin{proof} (1) is a straightforward combination of Theorems \ref{ACA_0_Pi^1_2} and \ref{omega_Pi^1_2}. (3) follows immediately from 5.6. 

For (2) we reason as follows:
\begin{flalign*}
\Sigma^1_1\text{-}\mathsf{AC}&\equiv \mathsf{\Sigma}^1_1\text{-}\mathsf{AC}_0+\mathbf{\Pi}^1_\infty\text{-}\mathsf{RFN}(\mathsf{\Sigma}^1_1\text{-}\mathsf{AC}_0)\\
&\equiv^{\mathbf{\Sigma}^1_1}_{\mathbf{\Pi}^1_2} \mathbf{\Pi}^1_2\text{-}\mathbf{R}^{\varepsilon_0}(\Sigma^1_1\text{-}\mathsf{AC}_0) \text{ by Theorem \ref{Schmerl-variant}}\\
&\equiv \mathbf{\Pi}^1_2\text{-}\mathbf{R}^{\varepsilon_0}(\mathsf{ACA}_0+\Sigma^1_1\text{-}\mathsf{AC}_0)\\
&\equiv_{\mathbf{\Pi}^1_2} \mathbf{\Pi}^1_2\text{-}\omega\mathbf{R}^{\varepsilon_0}(\mathsf{ACA}_0) \text{ by Lemma \ref{AC_cons_iter}}
\end{flalign*}
The conclusion follows since, by Theorem \ref{omega_iter_Pi^1_2}, $|\mathbf{\Pi}^1_2\text{-}\omega\mathbf{R}^{\varepsilon_0}(\mathsf{ACA}_0)|_{\mathbf{\Pi}^1_1}=\phi^+_{\varepsilon_0}(0)$.

For (4), we first note that:
\begin{flalign*}
\mathsf{ATR}&\equiv \mathsf{ATR}_0 + \mathbf{\Pi}^1_\infty\text{-}\mathsf{RFN}(\mathsf{ATR}_0)\\
&\equiv^{\mathbf{\Sigma}^1_1}_{\mathbf{\Pi}^1_1} \mathbf{\Pi}^1_1\text{-}\mathbf{R}^{\varepsilon_0}(\mathsf{ATR}_0) \text{ by Theorem \ref{Schmerl-variant}}
\end{flalign*}
We proceed to show that for any linear order $\alpha$ we have \begin{equation}\label{iter_ATR_0}|\mathbf{\Pi}^1_1\textrm{-}\mathbf{R}^\alpha(\mathsf{ATR}_0)|_{\mathbf{\Pi}^1_1}=\Gamma_\alpha.\end{equation} We prove it using L\"ob's theorem. 

We reason in $\mathsf{ACA_0}$ under the assumption $\mathsf{ACA}_0$-proves that for all for all linear orders $\alpha$ we have (\ref{iter_ATR_0}). We need to show that (\ref{iter_ATR_0}) holds. Indeed, for any linear order $\beta$ we have $\mathbf{\Pi}^1_1\textrm{-}\mathbf{R}^\alpha(\mathsf{ATR}_0)\vdash \mathsf{WO}(\beta)$ iff for some $x\in \alpha$ we have $$\mathsf{ATR}_0+\mathbf{\Pi}^1_1\textrm{-}\mathsf{RFN}(\mathbf{\Pi}^1_1\textrm{-}\mathbf{R}^{\mathsf{cone}(\alpha,x)}(\mathsf{ATR}_0))\vdash \mathsf{WO}(\beta).$$
But using our assumption about provability of (\ref{iter_ATR_0}) we see that $$\mathsf{ACA}_0\vdash \mathbf{\Pi}^1_1\textrm{-}\mathsf{RFN}(\mathbf{\Pi}^1_1\textrm{-}\mathbf{R}^\gamma(\mathsf{ATR}_0))\mathrel{\leftrightarrow} \mathsf{WO}(\Gamma_\gamma).$$
Thus $\mathbf{\Pi}^1_1\textrm{-}\mathbf{R}^\alpha(\mathsf{ATR}_0)\vdash\mathsf{WO}(\beta)$ iff for some $x\in \alpha$ we have $$\mathsf{ATR}_0+\mathsf{WO}(\Gamma_{\mathsf{cone}(\alpha,x)})\vdash \mathsf{WO}(\beta).$$ And since $|\mathsf{ATR}_0|_{\mathbf{\Pi}^1_2}=\Gamma^+$ the latter condition is equivalent to: $$\beta<\Gamma^+(\Gamma_{\mathsf{cone}(\alpha,x)})=\Gamma_{\mathsf{cone}(\alpha,x)+1}.$$ Thus (\ref{iter_ATR_0}) holds.
\end{proof}

Note that our technique actually delivers not just $\Pi^1_1$-ordinals but also reverse-mathematical well-ordering theorems.

\begin{corollary} Over $\mathsf{ACA}_0$ we have the following equivalences:
\begin{enumerate}
    \item \label{ACA_0^+_WOP}$\mathsf{ACA}_0^+$ is equivalent to the well-ordering principle $$\forall \alpha\Big(\mathsf{WO}(\alpha)\to  \mathsf{WO}\big(\phi_1(\alpha)\big)\Big).$$
    \item \label{ATR_0_WOP}$\mathsf{ATR}_0$ is equivalent to the well-ordering principle $$\forall \alpha\Big(\mathsf{WO}(\alpha)\to  \mathsf{WO}\big(\phi_\alpha(0)\big)\Big).$$
\end{enumerate}
\end{corollary}
\begin{proof} We reason in $\mathsf{ACA}_0$. By Theorem \ref{omega_to_iter} $\mathsf{ACA}_0^+$ is equivalent to $$\forall \alpha \Big( \mathsf{WO}(\alpha) \rightarrow \mathbf{\Pi}^1_1\mbox{-}\mathsf{RFN}\big(\mathbf{\Pi}^1_1\mbox{-}\mathbf{R}^\alpha(\mathsf{ACA}_0)\big)\Big).$$ And by Theorem \ref{iter_Pi^1_2} $\mathbf{\Pi}^1_1\mbox{-}\mathsf{RFN}(\mathbf{\Pi}^1_1\mbox{-}\mathbf{R}^{\alpha}(\mathsf{ACA}_0))$ is equivalent to $\mathsf{WO}(\varphi_1^{+\alpha})$. This yields \ref{ACA_0^+_WOP}.

We again reason in $\mathsf{ACA}_0$. By Theorem \ref{omega_to_iter}  $\mathsf{ATR}_0$ is equivalent to $$\forall \alpha \Big( \mathsf{WO}(\alpha) \rightarrow \mathbf{\Pi}^1_1\mbox{-}\mathsf{RFN}\big(\mathbf{\Pi}^1_2\mbox{-}\omega\mathbf{R}^\alpha(\mathsf{ACA}_0)\big)\Big).$$ The formula $\mathbf{\Pi}^1_1\mbox{-}\mathsf{RFN}\big(\mathbf{\Pi}^1_2\mbox{-}\omega\mathbf{R}^\alpha(\mathsf{ACA}_0)\big)$ by Theorem \ref{omega_iter_Pi^1_2} is equivalent to $\mathsf{WO}(\phi_{1+\alpha}(0))$. This yields \ref{ATR_0_WOP}.
\end{proof}

\bibliographystyle{alpha}
\bibliography{bibliography}{}

\end{document}